\documentclass[a4paper, 12pt]{article}

\usepackage{mathptmx}       
\usepackage{helvet}         
\usepackage{courier}        
\usepackage{graphicx}        
\usepackage{multicol}        
\usepackage[bottom]{footmisc}

\usepackage{url}
\usepackage{hyperref}
\usepackage{amsfonts}
\usepackage{amsmath}
\usepackage{amsthm}
\usepackage{amssymb}

\newcommand{\R}{\mathbb{R}}
\newcommand{\bs}{\it}

\newcommand{\N}{\mathbb{N}}

\newtheorem{prop}{Proposition}

\newtheorem{theorem}{Theorem}
\newtheorem{definition}{Definition}

\usepackage{authblk}
\title{This is some thing}
\author[1]{Stefano De Marchi \thanks{demarchi@math.unipd.it}}
\author[2]{Andrea Idda \thanks{a.idda1989@gmail.com}}
\author[3]{Gabriele Santin \thanks{santinge@mathematik.uni-stuttgart.de}}
\affil[1]{Department of Mathematics - University of Padova (Italy)}
\affil[2]{Banco Popolare - Verona (Italy)}
\affil[3]{IANS - University of Stuttgart (Germany)}

\title{A rescaled method for RBF approximation}

\begin{document}
\maketitle

\abstract{
In the recent paper \cite{DFQ14}, a new method to compute stable kernel-based interpolants has been 
presented. 
This \textit{rescaled interpolation} method combines the standard kernel interpolation with a 
properly defined rescaling operation, 
which smooths the oscillations of the interpolant. Although promising, this procedure lacks a 
systematic theoretical investigation.
Through our analysis, this novel method can be understood as standard kernel interpolation by means 
of a properly rescaled kernel.
This point of view allow us to consider its error and stability properties. 
}

\section{Introduction}
In the last decades radial basis functions have shown to be a flexible mathematical tool to solve 
scattered data
interpolation problems, to model neural networks, a meshfree method for solving differential 
equations, parametrizing shape and surfaces
and so on. Interested readers can refer, for example, to the comprehensive monographs \cite{B03,W02, 
FMcC15}, that give the necessary 
theoretical background and discuss many of the applications here enumerated.

In the paper \cite{DFQ14} has been presented an interpolation method for the construction of stable 
kernel-based interpolants, 
called \textit{rescaled interpolation}. It is a consistent local method that combines the standard 
kernel interpolation with a properly defined rescaling operation, 
which essentially smooths the oscillations of the interpolant. Although promising, the method lacks 
a systematic theoretical understanding.

After recalling some necessary notation, we present the method and prove that it is an instance of 
the 
well-known {\it Shepard's method}, when certain weight functions are used. 
In particular, as for the Shepard's one, it reproduces constant functions.

Second, it is possible to define a modified set of cardinal functions strictly related to the ones 
of the not-rescaled kernel. Through these
functions, we define a Lebesgue function for the rescaled interpolation process, and study its 
maximum - the Lebesgue constant - in different settings.

Also, a preliminary theoretical result on the estimation of the interpolation error is presented.

As an application, we couple our method with a partition of unity algorithm. This setting seems to 
be the most promising, 
and we illustrate its behavior with some experiments. 
The method has been also compared with the variably scaled kernel interpolation studied in 
\cite{BLRS15}.

We summarize briefly the paper structure. In the next section we introduce some basic definitions 
useful to understand
the results presented in the paper. Then, in successive section \ref{RI} we present the {\it 
rescaled localized} RBF interpolant and 
discuss some of its properties. In particular, in 
the successive subsection \ref{SubSecRI} we present the kernel-based approach to the rescaled 
interpolant, formalizing some
results already presented in \cite{DFQ14}. In the case of kernels depending 
on the shape parameter, an interesting property of the method is that the shape parameter can be 
chosen neither too small or too big.
In section \ref{Section3} we show that this interpolant is indeed a {\it Shepard's approximant}
and so it reproduces constants functions ``at glance''. Stability results of this construction are 
detailed in the successive subsection. 
The most promising application is the Partion of Unity Method (PUM). 
In Section \ref{Section4} we apply the rescaled interpolant to the PUM, showing then in the 
numerical experiments its effectiveness.
Finally in Section 5, we compare the behavior of the standand interpolant with respect to that of 
the rescaled and the {\it variably scaled}
ones, which was firstly studied in \cite{BLRS15}.
We will show that the combination
of PUM with the rescaled interpolant provides a stable method for interpolation.

\section{Uselful notations}
We start by recalling some notations useful for the sequel and necessary to understand the results 
that we are going to present.

Given a real Hilbert space ${\cal H}$ of functions from $\mathbb{R}^d$ to $\mathbb{R}$ with inner 
product $\langle \,,\, \rangle_{\cal H}$, a function 
$K : \Omega \times \Omega \longrightarrow \mathbb{R}$ with $\Omega \subset \R^d$, is called {\it 
reproducing kernel} for ${\cal H}$ 
if the two properties hold: 
\begin{enumerate}
 \item[(i)] $K(\cdot, {\bs x}) \in {\cal H}$ for all ${\bs x} \in \Omega$;
 \item[(ii)] $\langle f, K(\cdot, {\bs x}) \rangle_{\cal H}=f({\bs x})$ for all $f \in {\cal H}$ and 
${\bs x} \in \Omega$.
\end{enumerate}
The kernel is symmetric (by property (ii)) and positive definite by construction (cf. e.g. 
\cite[\S 2.3]{FMcC15} or \cite[\S 13.1]{F07}).

Let $\Phi: \mathbb{R}^d \rightarrow \mathbb{R}$ be a continuous function. $\Phi$ is called {\it 
radial} if there exists a continuous function
$\varphi: \R_0^+ \rightarrow \mathbb{R}$ such that $\Phi({\bs x})=\varphi(\|{\bs x}\|)$.

We are interested on {\it radial kernels}, i.e. kernels of the form $K({\bs x}, {\bs y})=\Phi({\bs 
x},{\bs y})=\varphi(\| {\bs x} - {\bs y} \|)$, which means {\it invariant under translation and 
rotation}, with $\|\cdot \|$ the Euclidean distance. Radial kernels are often referred as {\it 
Radial Basis Functions} of shortly RBF. Moreover, since $\varphi$ is a univariate function, we can 
call it the {\it basic radial function} and its values are function of a positive variable $r=\|{\bs 
x}\|$. 

Examples of ${\cal C}^\infty$ kernels are in Table \ref{table1} and kernels with finite smoothness 
in Table \ref{table2}. An interesting review of 
nonstandand kernels is the paper \cite{DeMS09} which presents many examples of kernels besides the 
classical ones presented in the Tables \ref{table1} and 
\ref{table2}.
We have written the kernels by introducing the parameter $\epsilon>0$ which is simply a {\it shape 
parameter}: for $\epsilon \to \infty$, the functions become more and more spiky while as $\epsilon 
\to 0$ they become flatter (see Figure \ref{fig1}).
In the case of Wendland's kernels the parameter $d$ denotes the maximal space dimension for which 
the functions are positive definite. With the symbol $\doteq$ we denote ``equal up to some constants 
factors''. The parameter $k$ indicates that they are ${\cal C}^{2k}(\R^d)$. For example for $d=1$, 
$l=3$ we have $\varphi_{1,1}(r)\doteq (1-r)^4_+(4r+1)$, which is the well-known compactly supported 
Wendland ${\cal C}^2$ function. Introducing also for Wendland functions a shape parameter 
$\epsilon>0$, we simply stretch the support to $[0,1/\epsilon]$
having $\delta=1/\epsilon$ as its diameter.
\begin{table}[!ht]
\centering
\begin{tabular}{|| c | c | c | c ||}
\hline 
kernel & parameter & support & name \\ 
\hline 
$\varphi(\epsilon r)=e^{-\epsilon^2 r^2}$ & $\epsilon > 0$ & $\R_0^+$ &  gaussian \\ 
$\displaystyle \varphi(\epsilon r)=(1+ \epsilon^2 r^2)^{-1} $ & $\epsilon > 0$ &  $\R_0^+$  & 
inverse quadrics \\ 
$\displaystyle \varphi(\epsilon r)=(1+ \epsilon^2 r^2)^{-1/2}  $ & $\epsilon > 0$ &  $\R_0^+$ & 
inverse multiquadrics \\
\hline
\end{tabular}
\caption{Examples of infinitely smooth kernels} \label{table1}
\end{table}

\begin{table}[!ht]
\centering
\begin{tabular}{|| c | c | c | c ||}
\hline 
kernel & parameter & support & name \\ 
\hline 
$\varphi(\epsilon r)=e^{-\epsilon r}$ & $\epsilon > 0$ &  $\R_0^+$ &  M0 \\ 
$\varphi(\epsilon r)=(1+\epsilon r) e^{-\epsilon r}$ & $\epsilon > 0$ &  $\R_0^+$ &  M2 \\ 
$\displaystyle \varphi_{d,0}(r) \doteq (1-r)^{l}_+$ & $l=\lfloor d/2+k+1 \rfloor$ &  $[0,1]$  & W0 
\\ 
$\displaystyle \varphi_{d,1}(r) \doteq (1-r)^{l+1}_+ ((l+1)r+1)$ & $l=\lfloor d/2+k+1 \rfloor$ &  
$[0,1]$  & W2 \\ 
\hline
\end{tabular}
\caption{Examples of kernels with finite smoothness. M0 and M2 are Mat\'ern kernels, 
W0 and W2 are Wendland kernels with smoothness $0$ and $2$ respectively. } \label{table2}
\end{table}

\begin{figure}[t]
\centering 
\includegraphics[scale=0.5]{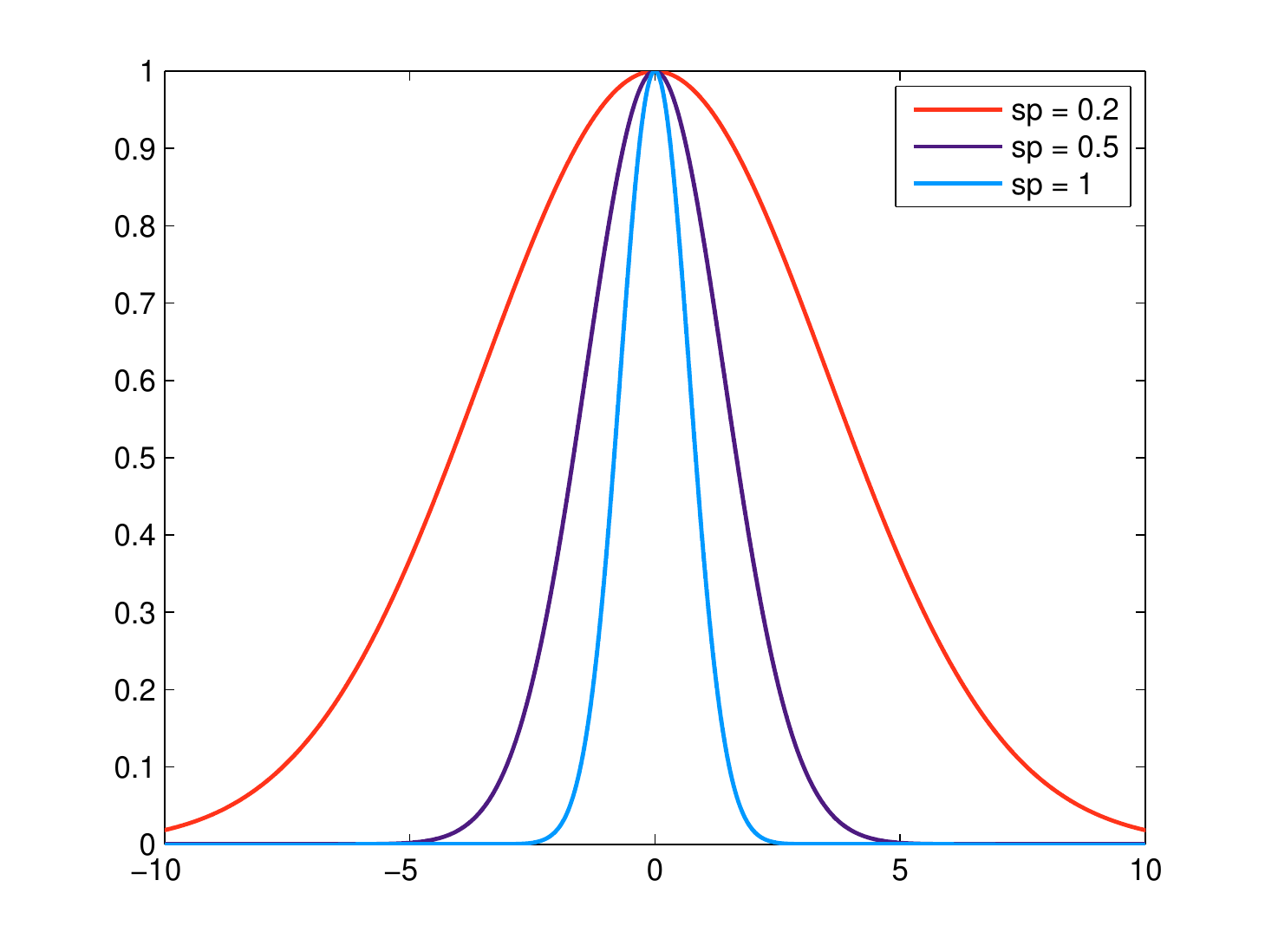}
\caption{The effect of changing the shape parameter on the gaussian kernel plotted in $[-10,10]$} 
\label{fig1}
\end{figure}
Now, given a function $f: \Omega \rightarrow \R$, a set $X=\{x_1,\ldots,x_N\} \subset \Omega$ of $N$ 
distinct points and the values
${\bs f}_X=(f(x_1), \ldots, f(x_N))^T$ of the function $f$ at the set $X$, we seek for interpolants 
of ${\bs f}_X$ of the form
\begin{equation} \label{eq1}
P_f({\bs x})=\sum_{i=1}^N c_i K({\bs x},{\bs x}_i),\, x\in\Omega.
\end{equation}
for kernels $K$ which are strictly positive definite and radial. This means that $P_f$ is a 
function in $H_K(X)={\rm span}\{K(\cdot, {\bs x}_i), \; i=1,\ldots, N\}$, formed by translates of 
$K$ at the point set $X$. 
The coefficients in (\ref{eq1}) are determined by imposing the interpolation conditions, 
which is equivalent to 
find the unique solution of the linear system $A  {\bs c}={\bs f}_X$ with $A_{i,j}=K({\bs x}_i,{\bs 
x}_j)$. 
Since the kernel $K$ is assumed to be strictly positive definite then the solution exists and is 
unique.

The Hilbert space in which the kernel $K$ is reproducing, is known as the associate {\it native 
space}. We will denote it by ${\cal N}_K$, instead of $\mathcal H$, to underline the dependence on 
the kernel. It is equipped by the scalar product $(\cdot, \cdot)_{\mathcal N_K}$, from which we get 
the {\it native space norm} $\|\;\|_{{\cal N}_K}$ (cf. e.g. \cite{F07}).

The interpolation process by kernels of functions $f \in {\cal N}_K$, 
gives pointwise errors of the form (see e.g. \cite[p. 174]{FMcC15} or \cite[p. 19]{DeMS09})
\begin{equation} \label{error}
 |f({\bs x})- P_f({\bs x})| \le C h_{X,\Omega}^\beta \|f\|_{{\cal N}_K}
\end{equation}
for some appropriate exponent $\beta$ depending on the smoothness of the kernel $K$, and 
$h_{X,\Omega}$ that denotes
the {\it mesh-size} (or fill-distance)
$h_{X,\Omega} = \max_{{\bs x} \in \Omega} \min_{{\bs x}_i \in X} \| {\bs x} -{\bs x}_i\|_2\,$. For 
functions belonging to bigger spaces
than the native space, suitable estimates
are based on {\it sampling inequalities} as discussed e.g. in \cite[\S 9.4]{FMcC15}.

\section{The rescaled interpolant} \label{RI}
In \cite{DFQ14} the authors have proposed a new compactly supported RBF interpolant with the aim of 
a more accurate interpolantion
even by using a small diameter for the support. More precisely, on the set of points $X$, we 
consider the constant function 
$g({\bs x})=1$ $\forall x\in\Omega$, and we denote by $P_{g}({\bs x})$ the corresponding 
kernel-based interpolant, that is
$$ P_{g}({\bs x}) = \sum_{i=1}^N d_i K({\bs x},{\bs x}_i)\,,$$
whose coefficients ${\bs d}=(d_1,\ldots,d_N)^T$ can be determined as the solution of the linear 
system $A{\bs d}={\bs 1}$, with
the vector ${\bs 1}$ of ones.

Then, the {\it rescaled} interpolant is
\begin{equation} \label{rescaled}
\hat{P}_{f}({\bs x})={P_{f}({\bs x}) \over P_{g}({\bs x})}={ \sum_{i=1}^N c_i K({\bs x},{\bs x}_i) 
\over \sum_{i=1}^N d_i K({\bs x},{\bs x}_i)} \,.
\end{equation}
As a simple illustrative example we want to interpolate $f(x)=x$ on the interval $[0,1]$ by using 
the W2 function at the points set 
$X=\{1/6,1/2,5/6\}$ with shape parameter $\varepsilon=5$. The function, the interpolant and the 
errors $|f(x)- P_{f}(x)|$, 
$|f(x)- \hat{P}_{f}({x})|$ are displayed in Figure \ref{Figure2}. 
\begin{figure}[!ht]
\begin{center}
 \includegraphics[width=5.4cm, height=4.4cm]{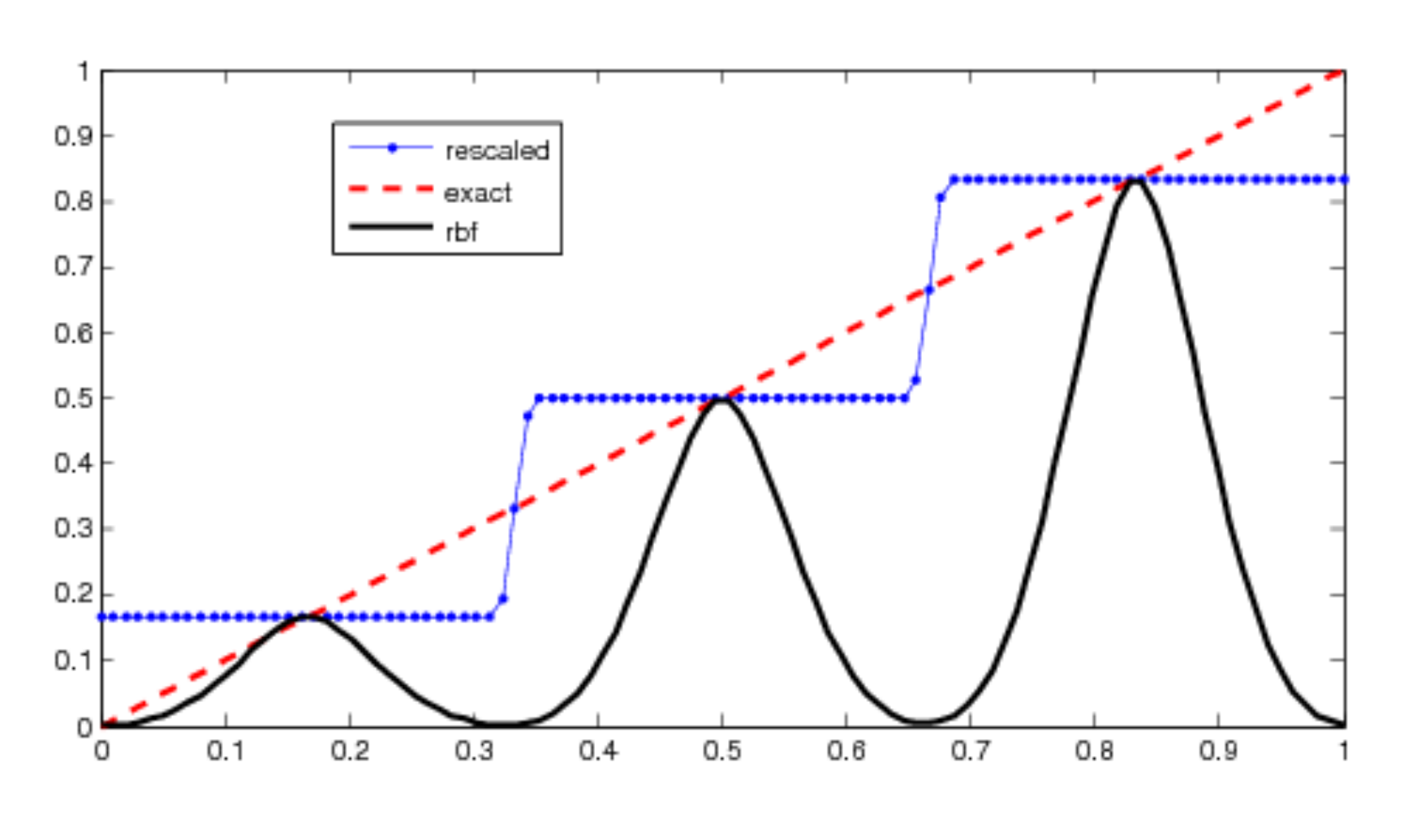}
 \includegraphics[width=5.6cm, height=4.5cm]{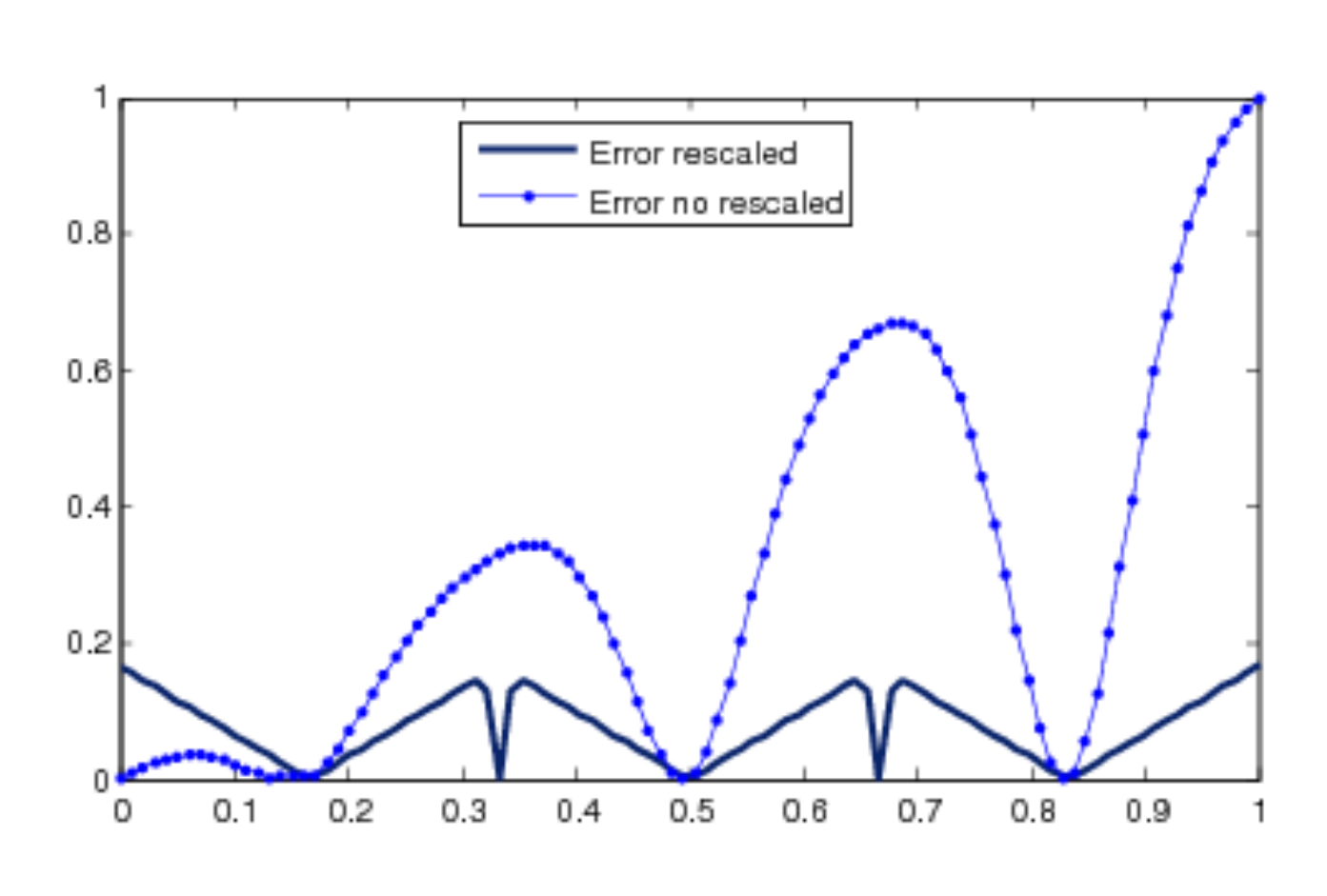}
 \caption{Left: the function $f(x)=x$ interpolated with the compactly supported W2 function and the 
rescaled interpolant at the data set $X=\{1/3,2/3,5/6\}$ 
 and $\varepsilon=5$. Right: the corresponding absolute errors.} \label{Figure2}
\end{center}
 \end{figure}
The shape parameter has been chosen as the reciprocal of the support radius of the corresponding 
basis function. Indeed, in this example the interpolant is the combination of {\it three} W2 radial 
functions all having radius of the support $r_j=1/5, \; j=1,2,3$. 
Refining the point set, that is considering the points set $X=\{0,1/6,1/3,1/2,2/3,5/6,1\}$, we get 
the results shown in Figure \ref{Figure2_1} 
showing the most interesting behavior of the rescaled interpolant, the property to reduce 
oscillations and so the interpolation error.
\begin{figure}[t]
\centering
 \includegraphics[width=5.8cm, height=4.7cm]{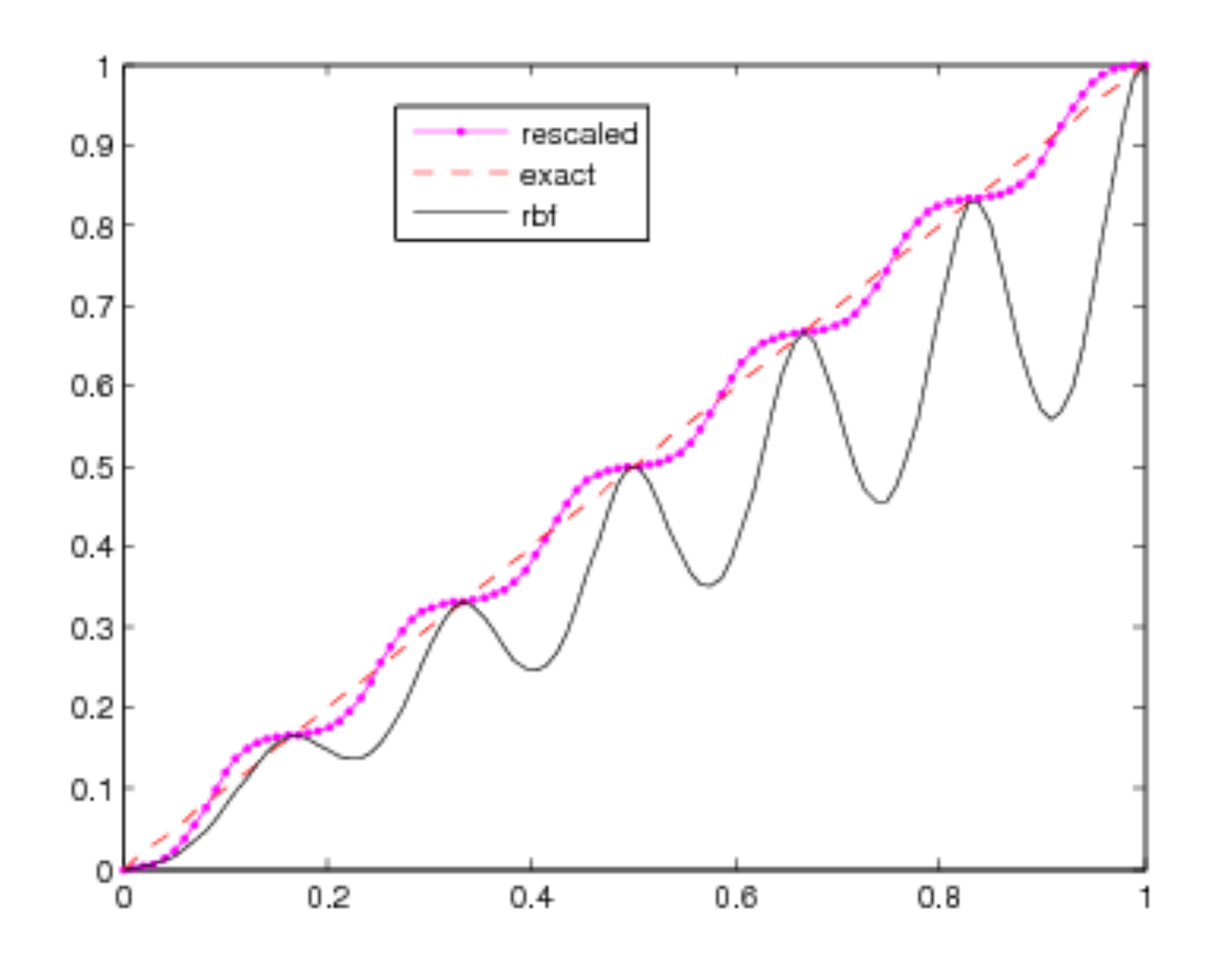}
 \includegraphics[width=5.8cm, height=4.8cm]{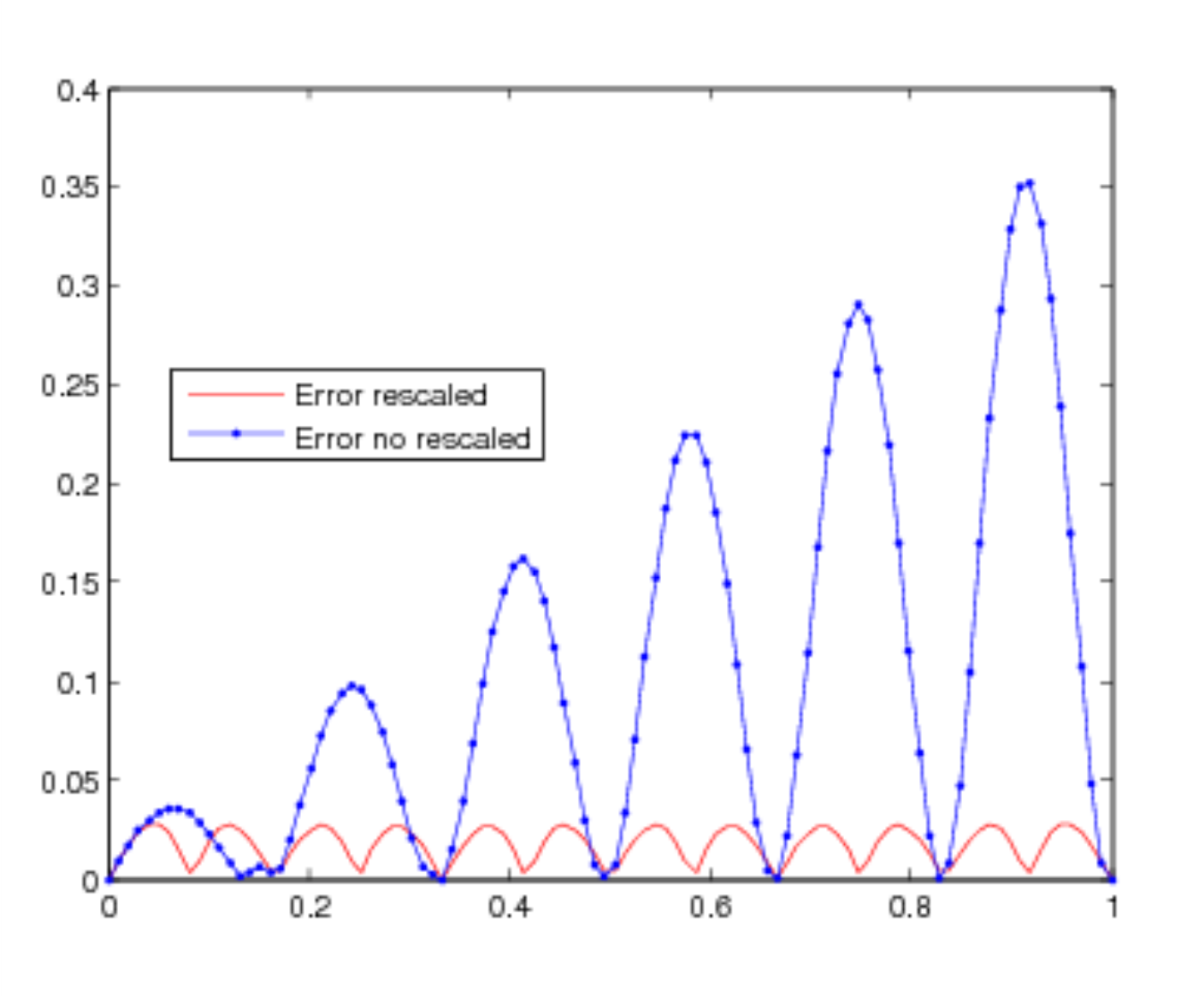}
 \caption{As in Figure \ref{Figure2} for the set $X=\{0,1/6,1/3,1/2,2/3,5/6,1\}$ and 
$\varepsilon=5$.} \label{Figure2_1}
\end{figure}

We show in Figure \ref{fig2} the RMSE (Root Mean Square Error) of the stationary interpolation with 
the rescaled interpolant (\ref{eq1}) w.r.t. the classical one (\ref{eq1}) on a grid of $25$ data 
points of the square $[0,1]^2$ at different values of the shape parameter, by using the W2 radial 
function for the 2d {\it Franke function} 
\begin{eqnarray} \label{f12d} 
  f(x,y) &=& \frac{3}{4}{\tt e}^{-\frac{1}{4}((9x-2)^2+(9y-2)^2)} + \frac{3}{4}{\tt 
e}^{-\frac{1}{49}(9x+1)^2-\frac{1}{10}(9y+1)} \\ \nonumber
        &  + & \frac{1}{2}{\tt e}^{-\frac{1}{4}((9x-7)^2+(9y-3)^2)} - \frac{1}{5}{\tt 
e}^{-(9x-4)^2-(9y-7)^2}.
        \end{eqnarray}
Similar results can be obtained by using different radial basis functions as studied in 
\cite{DFQ14,Idda}.
\begin{figure}[t]
\centering
  \includegraphics[scale=0.5]{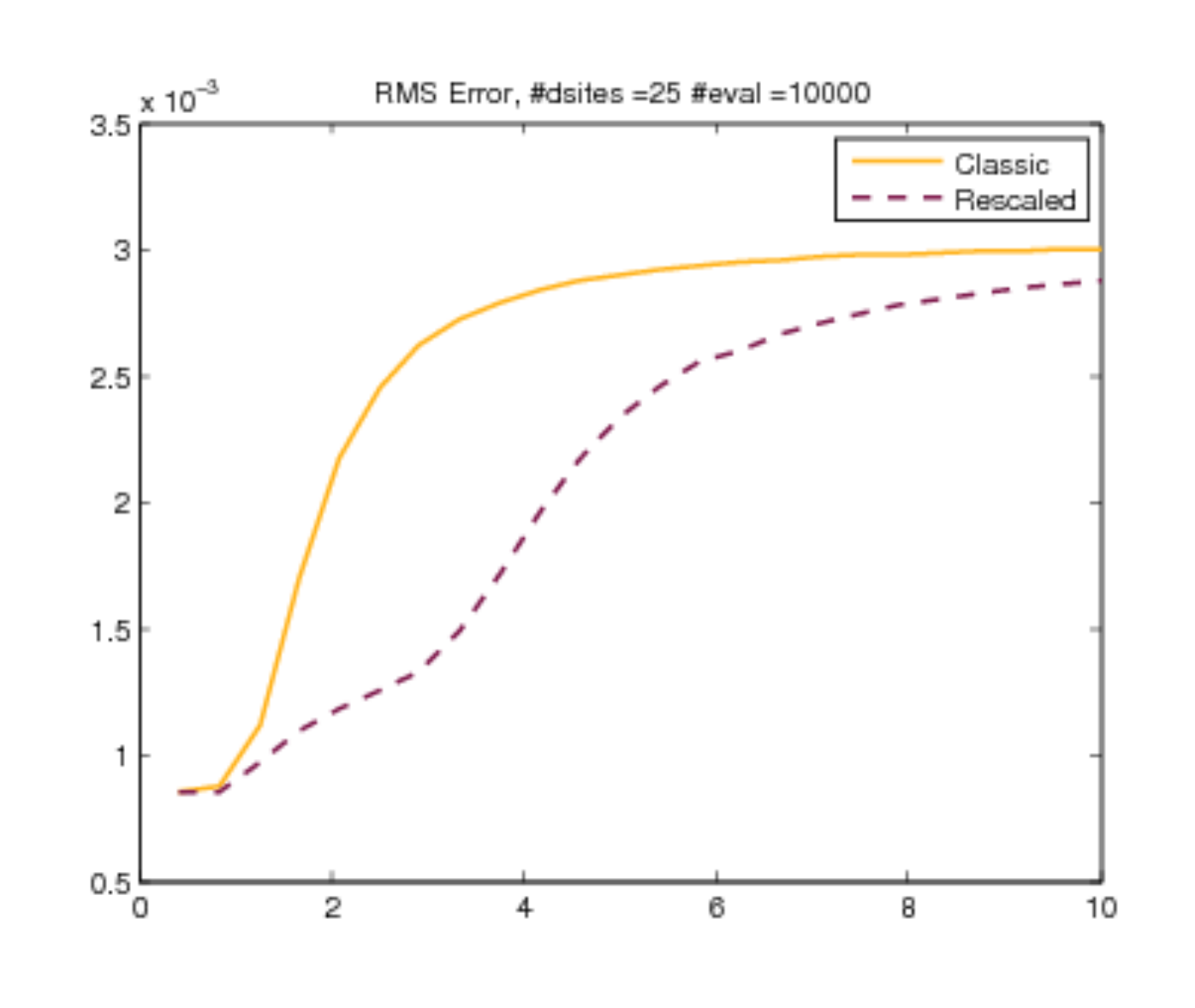}
  \caption{RMSE behavior at different values of the shape parameter for interpolation of the Franke 
function with W2 radial function.} \label{fig2}
\end{figure}

{\bf Remarks.} On looking to the way in which the interpolant is constructed and the previous 
figures we can observe
\begin{itemize}
 \item The interpolant is smooth even for small radii of the support.
 \item Thanks to the normalization introduced in (\ref{rescaled}), the method can choose for each 
$x_m$ a strategy to locally select the
 shape parameter $\varepsilon$ in order to take into account the data points distribution. In the 
paper \cite{DFQ14} the choice is made
 so that the local radius of the compactly supported kernel gives a constant number of neighbors. 
This strategy fails when the points are uniformly distributed while
 gives much better results when the points are not equidistributed. In this second case we can in 
fact consider different radii and neighbor points.
\end{itemize}

\subsection{The rescaled kernel}\label{SubSecRI}
Since our interest is the study of the rescaled interpolant as a 
new kernel which has an associated native space we start by the following observation. 
An interpolation process consists in approximating the function $f$ by its interpolant, say $P_f$, 
that is, for a constant $k_0\in\R$,
$$P_f({\bs x}) \approx f({\bs x}), \;\; \forall \; {\bs x} \in \Omega\,.$$
where equality holds for $x \in X$. Equivalently we can say that
$$P_f({\bs x}) - f({\bs x}) \approx k_0, \;\; \forall \; {\bs x} \in \Omega\,.$$ 
Assuming that $f \ne 0$, then 
\begin{equation} \label{eq2}
{P_f({\bs x}) - f({\bs x}) \over f({\bs x})} \approx{ k_0 \over f({\bs x})} =k_0, \;\; \forall \; 
{\bs x} \in \Omega\,,
\end{equation} 
the last equality holds when $f\equiv 1$.
Hence, assuming that $f=k_1$, then from (\ref{eq2}) we get
\begin{equation} \label{eq3}
{P_{k_1}({\bs x}) } - k_1 \approx k_0, \;\; \forall \; {\bs x} \in \Omega\,,
\end{equation} 
where we used $k_0$ to describe the new constant of the right hand side.
\begin{prop}
The relation $\approx$ induced by (\ref{eq3}) is an equivalence relation.
\end{prop}
\begin{proof} Let $f,g,h: \Omega \rightarrow \R$ and $X$ be a set of distinct points of $\Omega$.  
\begin{itemize}
\item Reflexivity. Since $f=f$, then $f(\Omega)=f(\Omega)$ implies $f(X)=f(X)$ and so $f \approx f$.
\item Symmetry. Let $f\approx g$, then $f(X)=g(X)$  that can be read from right to left and so 
$g \approx f$.
\item Transitivity. Let $f\approx g$ and $g \approx h$. This means $f(X)=g(X)$ and $g(X)=h(X)$. 
Hence $f(X)=h(X)$, that is $f \approx h$.
\end{itemize}
This concludes the proof. \qed
\end{proof}
 
We can use the transitive property of $\approx$ to combine (\ref{eq2}) and (\ref{eq3}) to get
\begin{eqnarray*}
 {P_{f}({\bs x}) \over f({\bs x})} -k_1 &\approx& {P_{k_1}({\bs x}) } - k_1\\
  {P_{f}({\bs x}) \over f({\bs x})} &\approx&  {P_{k_1}({\bs x}) } \\
  {P_{f}({\bs x}) \over P_{k_1}({\bs x})} &\approx&  f({\bs x})\,.
  \end{eqnarray*}
Therefore, functions of the form $\displaystyle {P_{f}({\bs x}) \over P_{k_1}({\bs x})}$ are 
rescaled interpolants.

In our setting, we can notice that both $P_{f}({\bs x})$ and $P_{g}({\bs x})$ are constructed by 
using the kernel $K$ 
with associate native space ${\cal N}_K$. In order to identify
the native space associated to the rescaled interpolant, we may proceed as follows. Letting
\begin{equation} 
\hat{P}_{f}({\bs x})={P_{f}({\bs x}) \over P_{g}({\bs x}) } \,,
\end{equation}
we may introduce a new kernel, say $K_r$, associated to the rescaled interpolant, from which we 
will characterize the associated native space ${\cal N}_{K_r}$.

Observing that
\begin{equation} \label{rescaled1}
\hat{ {P}}_{f}( \bs{x}) = \frac{ {P}_f( \bs{x})}{ {P}_{g}( \bs{x})} = \sum_{j=1}^{N}c_{j}\frac{K( 
\bs{x}, \bs{x}_{j})}{\sum_{i=1}^{N}d_{i}K( \bs{x}, \bs{x}_{i})}
\end{equation}
and recalling that the denominator is the interpolant of the constant function $g({\bs x})=1$, we 
have 
\begin{equation*}
\hat{ {P}}_{f}( \bs{x}) = \sum_{j=1}^{N}c_{j} \left[ \frac{K( \bs{x}, \bs{x}_{j})}{\displaystyle 
\sum_{i=1}^{N}d_{i}K( \bs{x}, \bs{x}_{i}) \, \sum_{i=1}^{N}d_{i}K( \bs{x}_{j}, \bs{x}_{i})}\right]
\end{equation*}
Denoting $q( \bs{x})=\sum_{i=1}^{N}d_{i}K( \bs{x}, \bs{x}_{i})$, then the square brakets can be 
re-written as the set of functions
\begin{equation*}
\frac{K( \bs{x}, \bs{x}_{j})}{q( \bs{x})\cdot q( \bs{x}_{j})}, \quad j=1,\dots,N
\end{equation*}
which can be interpreted as a {\it new basis} for the rescaled interpolant.

This theorem finds application in our setting.
\begin{theorem}[cf. \cite{A50}]
Let $K : \Omega \times \Omega \rightarrow \R$ be a (strictly) positive definite kernel. 
Let $s: \Omega \rightarrow \R$ be a continuous and nonvanishing function in $\Omega$. Then 
\begin{equation}
K_s({\bs x}, {\bs y})=s({\bs x}) s({\bs y}) K({\bs x}, {\bs y})
\end{equation}
is (strictly) positive definite.
\end{theorem}
In fact, letting $P_{g}$ the interpolant of the constant $g \equiv 1$ which is by construction
continuous in $\Omega$, if it is also non-vanishing, then $s=1/P_{g}$. 
But this property follows from the general error estimation (\ref{error}), since we can find a 
set of points $X$ 
so that $\| P_{{g}}({\bs x}) - g\|_\infty < g$
that implies $P_{g}({\bs x}) \ne 0, \; \forall {\bs x} \in \Omega$. 

It follows that we can consider $s=1/P_{g}$, which is continuous, non-vanishing on $\Omega$ and 
consider
the {\it rescaled kernel}
\begin{equation}
K_r({\bs x}, {\bs y})={1 \over P_{{g}}({\bs x}) } { 1 \over P_{g}({\bs y})} K({\bs x}, {\bs y})
\end{equation}
which turns out to be (strictly) positive definite and we will denote its associate native space by 
${\cal N}_{K_r}$. 

This discussion shows that the rescaled kernel $K_r$ is a kernel approximation process which 
is well-posed and preserves the properties of the kernel $K$.
In Appendix \ref{app:native} we propose a preliminary analysis aimed to show 
possible connections between the native spaces of the kernel $K$, ${\cal N}_K$
and that of the rescaled kernel $K_r$, ${\cal N}_{K_r}$.


Moreover, by re-writing the rescaled interpolant as 
 \begin{eqnarray*}
\hat{P}_f({\bs x})&=&\sum_{j=1}^N c_j K_r({\bs x}, {\bs x}_j) \\
&=& \sum_{j=1}^N c_j \left\{ {1 \over \sum_{i=1}^N d_i  K({\bs x}, {\bs x}_i)} \,  {1 \over 
\sum_{i=1}^N d_i  K({\bs x}_i, {\bs x})}\right\}
 { K({\bs x}, {\bs x}_j)} \\
&=& \sum_{j=1}^N c_j  {1 \over \sum_{i=1}^N d_i  K({\bs x}, {\bs x}_i)}  \, K({\bs x}, {\bs x}_j)\,,
 \end{eqnarray*}
since ${1 \over \sum_{i=1}^N d_i  K({\bs x}, {\bs x}_i)} =1\,,\;\forall \, {\bs x} \in \Omega$ is 
the interpolant of the function $g={1}$. This construction allows to prove formally that 
the rescaled interpolant reproduces the constants.
\begin{theorem}
Let $K$ be a strictly positive definite kernel, $f: \Omega \rightarrow \R \backslash \{0\}$ such 
that 
$f({\bs x})=a, \; a \in \R \backslash \{0\}$ and $X=\{x_1,\ldots, x_N\} \subset \Omega$. Then
the associated rescaled interpolant (\ref{rescaled}) (or equivalently (\ref{rescaled1})) 
is such that $\hat{P}_{r,a}({\bs x})=a, \; \;\forall {\bs x}$.
\end{theorem}
\begin{proof} The interpolation conditions give the linear system
$$ A_{r} {\bf c}={\bf a}$$ 
where ${\bf a}=(a,\ldots ,a)^T$ and $A_r$ denotes 
the collocation matrix w.r.t. the rescaled basis of the function $f(x)=a$. The previous system can 
be written as
\begin{eqnarray*}
a \cdot \left({1 \over a} A_r {\bf c}\right) &=& a \cdot {\bf 1}\,,\\
{1 \over a} A_r {\bf c} & =& {\bf 1}\,.
\end{eqnarray*}
where ${\bf 1}=(1,\ldots,1)^T$. Hence, denoting as $\hat{P}_{r,a}$ the rescaled interpolant of the 
constant $a$
\begin{equation}\label{eqqq1} 
\hat{P}_{r,a}(\cdot)=a \cdot \hat{P}_{g}(\cdot) = 
a \cdot {{P}_{g}(\cdot) \over {P}_{g}(\cdot)} = a\,,\end{equation}
as required. \qed
\end{proof}
\vskip 0.1in
Obviously the previous results holds for $a=0$. 
This comes immediately from (\ref{eqqq1}).

\section{Rescaling is equivalent to the Shepard's method} \label{Section3}
We take into account here a different point of view. We firstly compute the cardinal function form 
of the 
interpolatant, then we show the connection with the Shepard's method 
(see e.g. \cite[\S 23.1]{F07}) and provide a stability bound based on the Lebesgue constant. We need 
to recall
the following result (cf e.g. \cite[\S 14.2]{F07} or \cite{WS93}).
\begin{prop}
For any set $X_N=\{x_1, \dots, x_N\}\subset\Omega$ of pairwise distinct points, there exists a 
unique \texttt{cardinal basis}
$U=\{u_j\}_{j=1}^N$ of the $\mbox{span}\{K(\cdot, x), x\in X_N\}$, i.e. a set of functions such 
that 
$u_j(x_i) = \delta_{i j},\;\;1\leq i, j\leq N.$
\end{prop}
Using the basis $U$, the standard interpolant of a function $f\in\cal H$ can be written in the form
$P_{f} = \sum_{j=1}^N f(x_j) u_j, $
and the interpolant of the function $g\equiv 1$ reads as $P_{g}  = \sum_{j=1}^N u_j$. The rescaled 
interpolant of $f$ then takes the form
$$
\hat P_{f} = \frac{\sum_{j=1}^N f(x_j) u_j}{\sum_{k=1}^N u_k} = \sum_{j=1}^N f(x_j) 
\frac{u_j}{\sum_{k=1}^N u_k} =: \sum_{j=1}^N f(x_j)
\hat{u}_j,
$$
where we introduced the $N$ functions $\hat{u}_j :=u_j/\left(\sum_{k=1}^N u_k\right)$. These 
functions are still cardinal functions, since $\hat{u}_j(x_i) =
\delta_{ij}$, but they do not belong to the subspace $\mbox{span}\{K(\cdot, x), x\in X_N\}$, in 
general. But they form a partition
of unity, in the sense that, for all $x\in\Omega$, we have
$$
\sum_{j=1}^N \hat{u}_j(x) = \sum_{j=1}^N\frac{u_j(x)}{\sum_{k=1}^N u_k(x)} = 
\frac{\sum_{j=1}^Nu_j(x)}{\sum_{k=1}^N u_k(x)}= 1.
$$
This construction proves the following result.
\begin{prop}
The rescaled interpolation method is a Shepard's method, where the weight functions are defined as 
$\hat{u}_j =u_j/\left(\sum_{k=1}^N
u_k\right)$, $\{u_j\}_j$ being the cardinal basis of $\mbox{span}\{K(\cdot, x), x\in X \}$.
\end{prop}
{\bf Remarks}
\begin{itemize}
\item Looking at Figures \ref{Figure2} and \ref{Figure2_1}, we notice the typical ``flat-spot'' 
behaviour of the Shepard's approximation.
\item Although this connection allows to relate our work with other existing methods, we remark that 
there are some limitations in the present
approach. On one hand, the Shepard's method in its original formulation, is able to reproduce 
constant functions. On the other hand,
the main reason to consider Shepard's methods relies on the easy computation of the weight functions 
$\{\hat{u}_j\}_j$, which are usually
constructed by solving a small linear system instead of a full interpolation problem. In our case,
instead, the computation of the weights requires the computation of the cardinal basis, that is the 
solution of 
the full interpolation problem.
\end{itemize}
From this construction we can easily derive stability bounds for the rescaled interpolation 
process. 
In fact, as happens in polynomial interpolation by using the cardinal functions, we can define the 
Lebesgue function 
$
\Lambda_N(x) := \sum_{j=1} ^N |u_j(x)|,
$
and its maximum over $\Omega$, 
$
\lambda_N := \|\Lambda_N\|_{\infty}.
$
that is the {\it Lebesgue constant} which controls the stability of the interpolation process. In 
fact, for
any $x\in\Omega$ 
$$
|P_{f} (x)| = \left|\sum_{j=1}^N f(x_j) u_j(x)\right|\leq\left(\sum_{j=1}^N |u_j(x)|\right)  
\|f\|_{\infty, X} \le
\lambda_N \, \|f\|_{\infty, X}.
$$
Extending the setting to our case, and by using the rescaled cardinal functions $\{\hat{u}_j\}_j$ 
instead of the classical cardinals, we can write
$$
\hat\Lambda_N(x) := \sum_{j=1} ^N |\hat{u}_j(x)|,\;\;\hat\lambda_N := \|\hat\Lambda_N\|_{\infty},
$$
which gives the stability bound
$$
\|\hat P_{f}\|_{\infty} \le \hat\lambda_N  \|f\|_{\infty, X}.
$$
Hence, to quantify the stability gain of the rescaled interpolation process over the standard one, 
we can simply compare the behavior of
$\hat\lambda_N$ and $\lambda_N$. Numerical experiments showing this comparison 
are presented in Section \ref{NE}.

\section{Application to PUM} \label{Section4}
Given the domain $\Omega\subset \R^d$, we consider its partition $\{ \Omega_k \subset \Omega, \; 
k=1,\ldots, n\}$ 
with $\Omega_k$ that possibly overlap,
such that $\Omega \subseteq \cup_{k=1}^n \Omega_k$. We then consider 
compactly supported functions $w_k$ with ${\tt supp}(w_k)\subseteq \Omega_k$, forming a partition of 
$\Omega$, that is 
\begin{equation} \label{PoU} \sum_{k=1}^n w_k({\bs x})=1\,,\;\; \forall {\bs x} \in 
\Omega\,.\end{equation}
Then we construct of a {\it local interpolant}, $p_k$, in RBF form
\begin{equation}\label{LocalApprox} p_k({\bs x}; X_k) =\sum_{j=1}^{n_k} c_j^{(k)} \Phi_j^{(k)}({\bs 
x})\,,\end{equation}
where $X_k$ is a set of distinct points of $\Omega_k$ having $n_k=|X_k|$ as its cardinality and 
$\Phi^{(k)}$ 
the RBF kernel at $\Omega_k$. The global interpolant on $\Omega$ can be written as
\begin{equation} \label{GPUM}
P_{f}({x})=\sum_{k=1}^n p_k({\bs x}; X_k) w_k({\bs x}), \; \;{\bs x} \in \Omega\,.  
\end{equation}
If the local fit interpolates at a given data points, 
that is $p_k({x}_l)=f({ x}_l)$, then thanks to the partition of unity property (\ref{PoU}) 
we can conclude that the global fit is also interpolating at the same point
$$ P_{f}({\bf x}_l)=\sum_{k=1}^n p_k({\bs x}_l; X_k) w_k({\bs x}_l)= \sum_{k=1}^n f({\bs x}_l) 
w_k({\bs x}_l)=f({\bs x}_l)\,.$$
We can apply the rescaled interpolant to this framework as follows
\begin{itemize}
 \item by applying the rescaling to the global interpolant (\ref{GPUM});
 \item or by applying the rescaling to every local interpolant (\ref{LocalApprox}).
\end{itemize}
The first approach is equivalent to apply the PUM for interpolating the constant function 1. 
Hence, it makes sense to rescale {\it every} local interpolant.
The application of the rescaling to every local interpolant of the form (\ref{LocalApprox}) 
gives a global rescaled interpolant of the form
\begin{equation}
 P_f({\bs x})=\sum_{k=1}^n \hat{R}_k({\bs x}; X_k) w_k({\bs x}), \; \;{\bs x} \in \Omega\,.  
\end{equation}
with
$$ \hat{R}_k({\bs x}; X_k)= \sum_{j=1}^{n_k} c_j^{(k)}{ \Phi_j^{(k)}({\bs x}) \over P_1^{(k)}({\bs 
x})} = 
\sum_{j=1}^{n_k} c_j^{(k)}{ \Phi_j^{(k)}({\bs x}) \over \sum_{l=1}^{n_k} d_l^{(k)} \Phi_l^{(k)}({\bs 
x}) }\,,
$$
where the coefficients $d_l^{(k)}$ are chosen so that $\sum_{l=1}^{n_k} d_l^{(k)} \Phi_l^{(k)}({\bs 
x}) =1, \; \; \forall {\bs x} \in X_k\,.$


\section{Numerical examples} \label{NE}
The examples here presented aim to support the theoretical aspects so far analyzed and to show the 
performance of the method compared with the classical approach and the {\it variably scaled} 
approach. 
\subsection{Comparison of the standard and rescaled Lebesgue functions}
In these experiments we compare the standard Lebesgue function with that of the rescaled one. 
Since we need to directly compute the cardinal functions, that is a seriously unstable operation, we 
try to keep the example 
as simple as possible to avoid comparing the effect of the ill-conditioning. 
To this end, we work in $\Omega=[-1, 1]\subset \R$ and with a small number of fixed points. 
We use the Gaussian kernel (global and ${\cal C}^\infty$) and the Wendland W2 kernel, 
which is compactly supported and ${\cal C}^2$. The computation of the standard and rescaled Lebesgue 
functions is been
repeated for $\varepsilon = 0.5, \,1,\, 4,\, 8$ (Gaussian kernel) and $\varepsilon = 0.5,\, 1,\, 2, 
\, 4$ (Wendland kernel), as shown
in Figures \ref{fig:GaussLeb} and \ref{fig:WenLeb}. In the latter case the behavior of the Lebesgue 
function does not change for bigger values of the shape parameter that is why we stopped at 
$\varepsilon=4$. 

\begin{figure}[!h]
\centering
\begin{tabular}{cc}
\includegraphics[width=0.45 \textwidth]{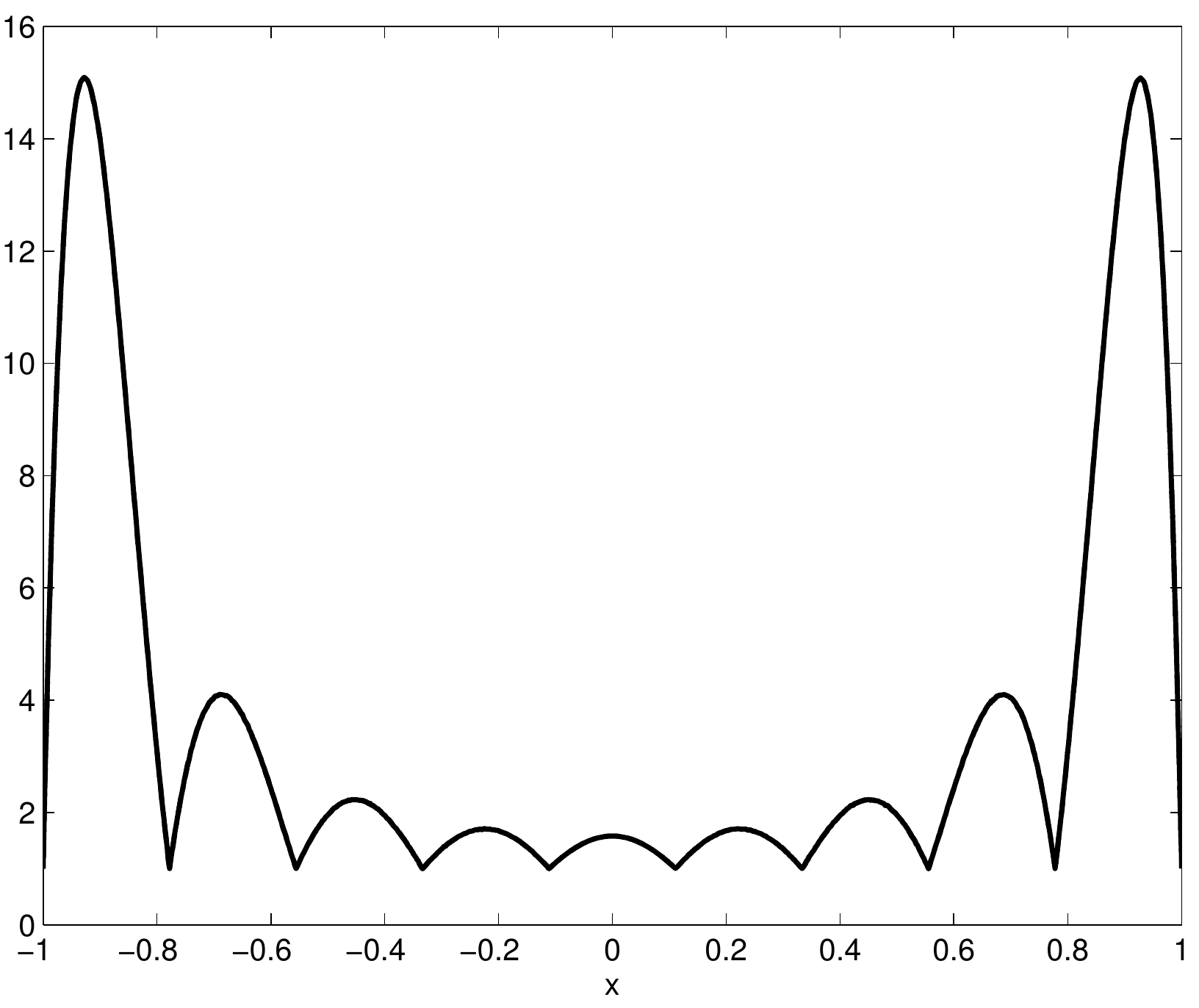}&
\includegraphics[width=0.45 \textwidth]{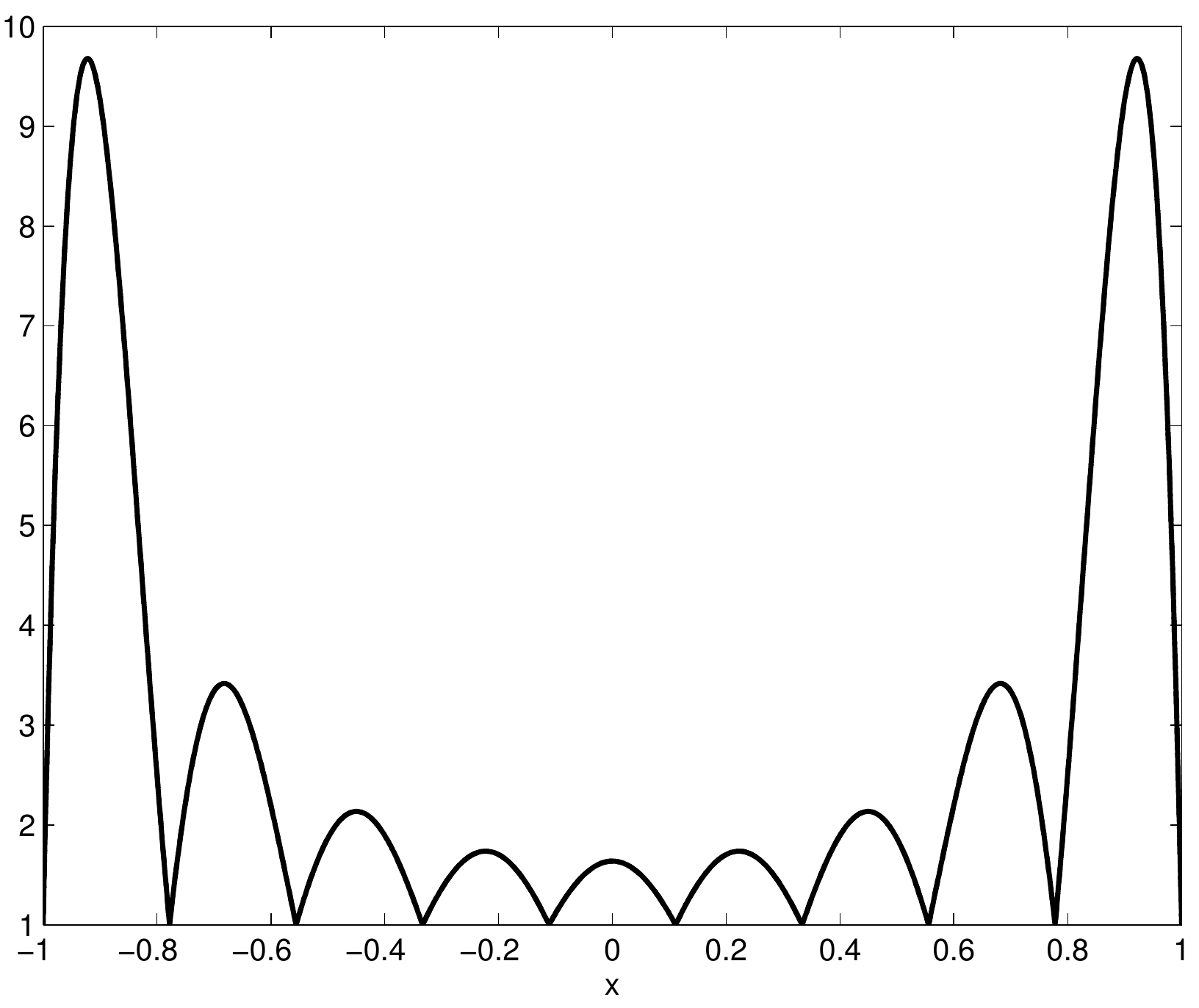}\\
\includegraphics[width=0.45 \textwidth]{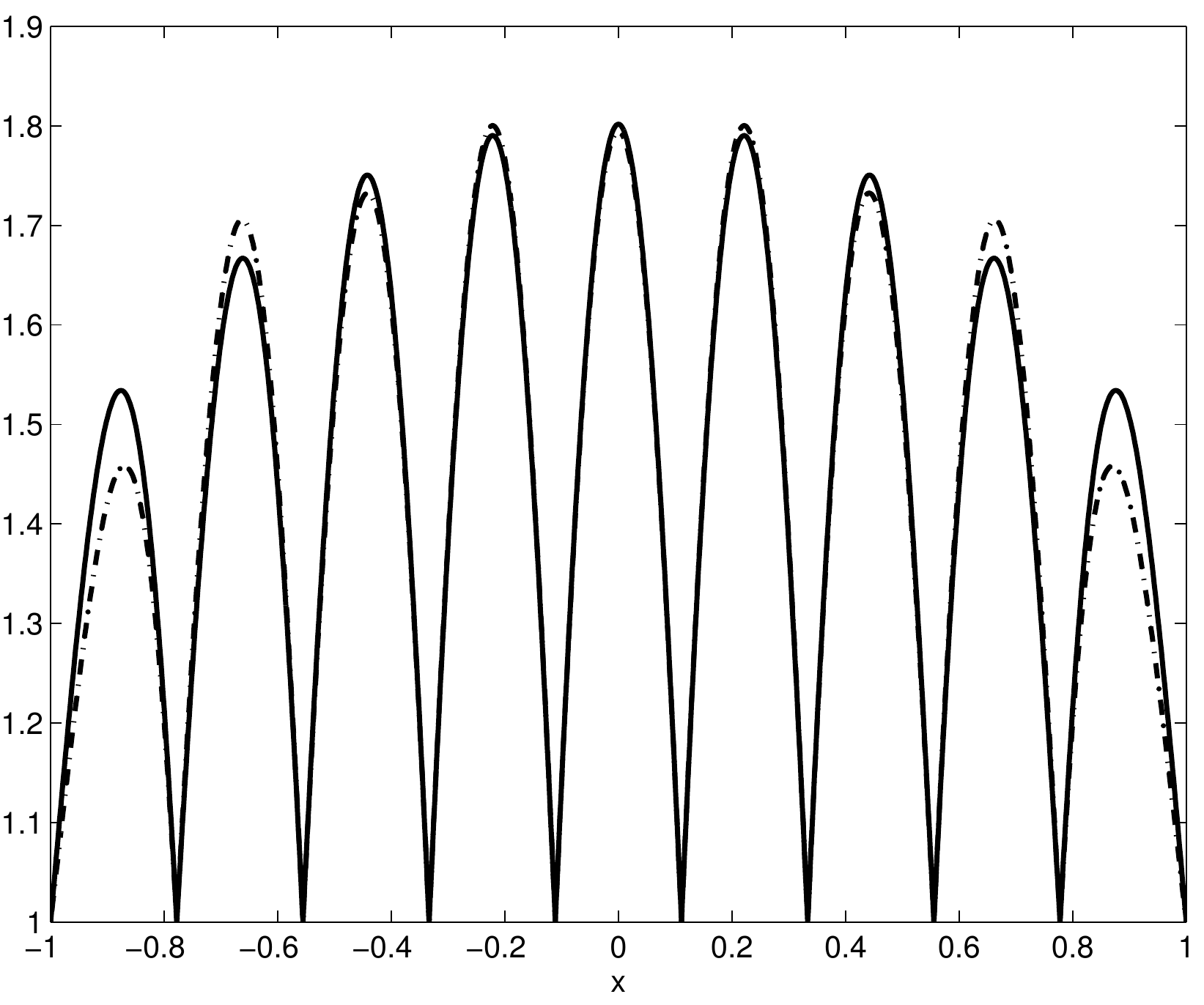}&
\includegraphics[width=0.45 \textwidth]{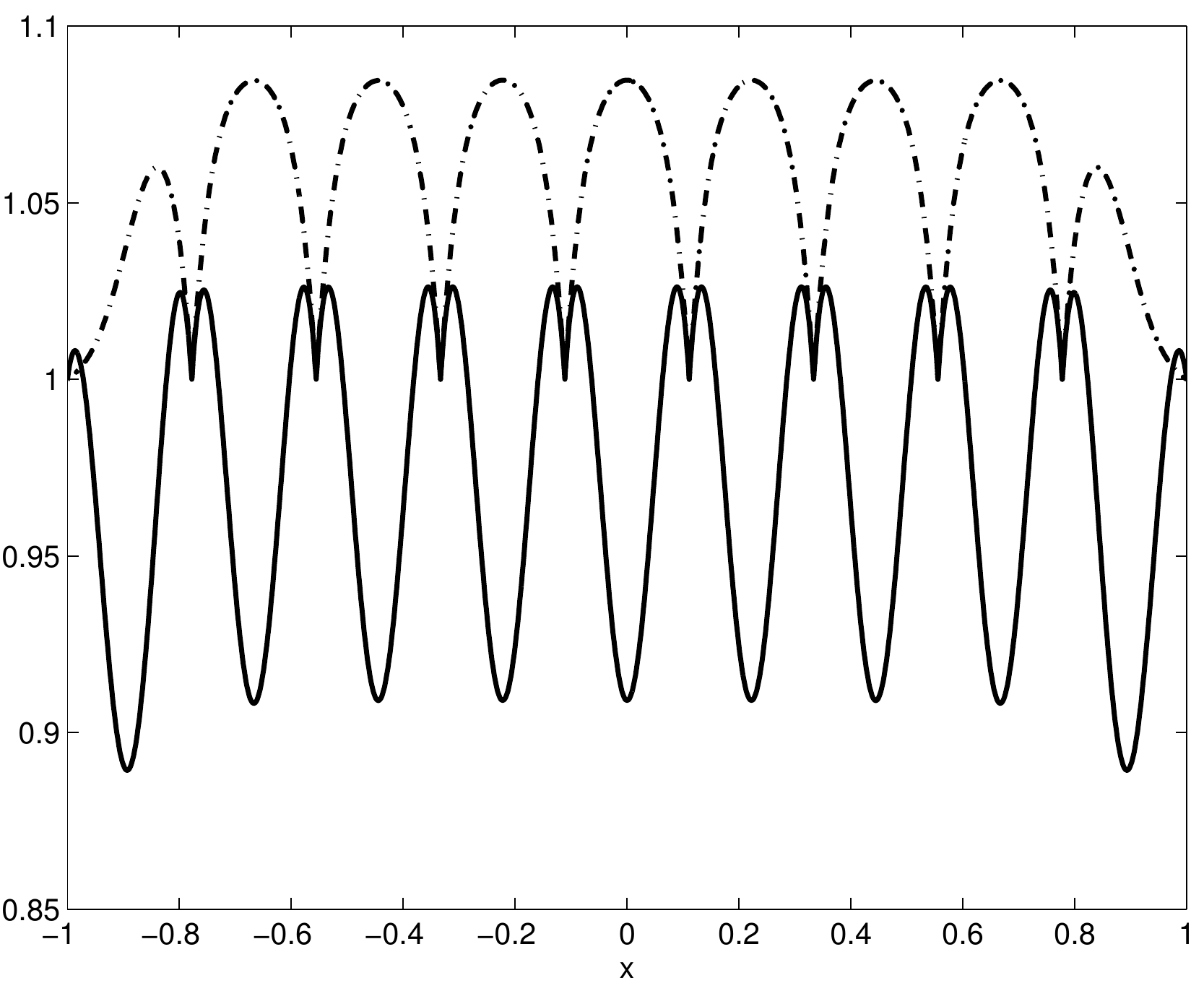}
\end{tabular}
\caption{Comparison between the standard Lebesgue function (solid line) and the rescaled 
Lebesgue function (dotted line) for the Gaussian kernel on $10$ equally spaced points of $[-1,1]$ 
with different $\varepsilon$. 
From top left to bottom right, $\varepsilon = 0.5, \,1,\, 4,\, 8$.}\label{fig:GaussLeb}
\end{figure}

\begin{figure}[!h]
\centering
\begin{tabular}{cc}
\includegraphics[width=0.45 \textwidth]{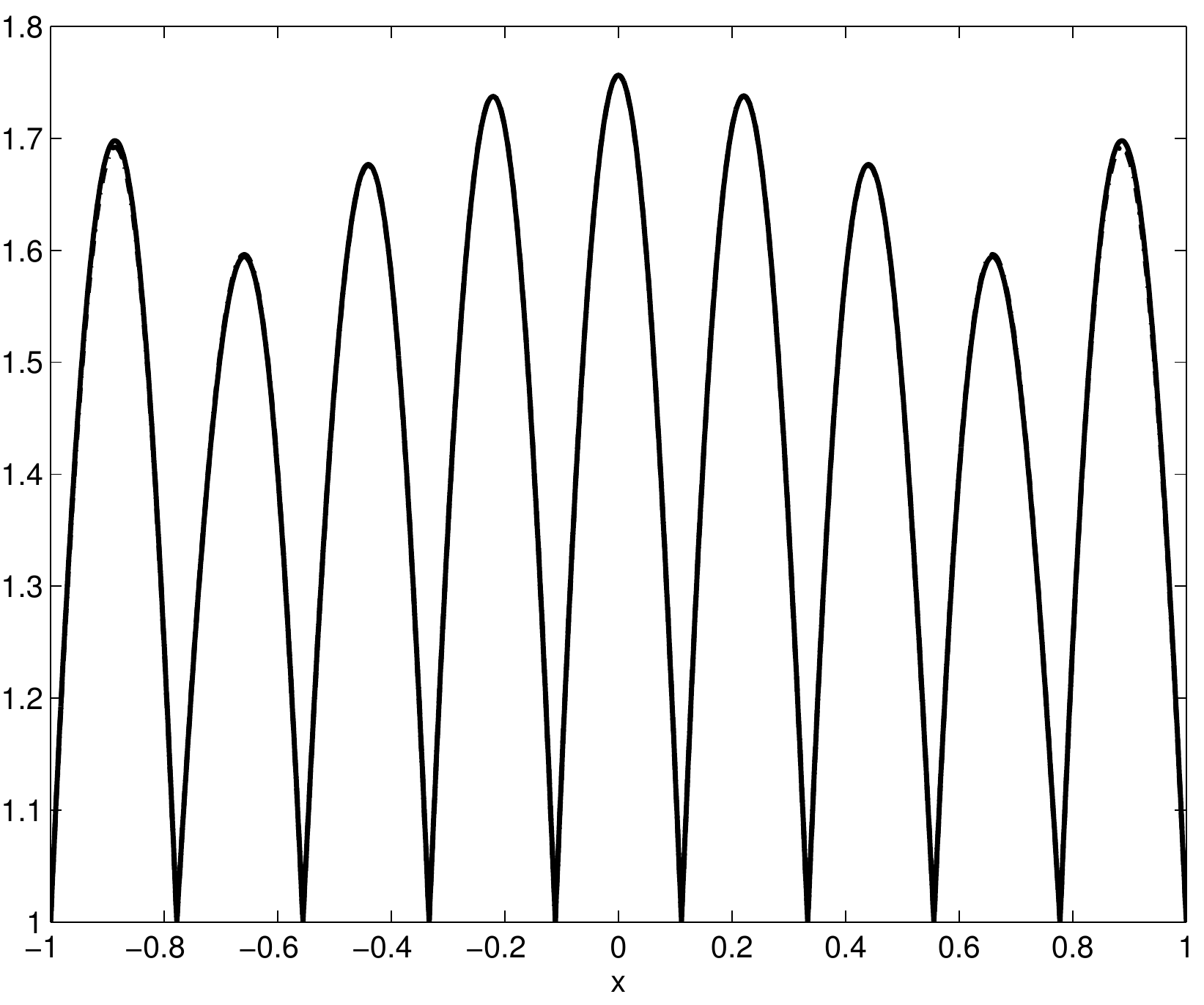}&
\includegraphics[width=0.45 \textwidth]{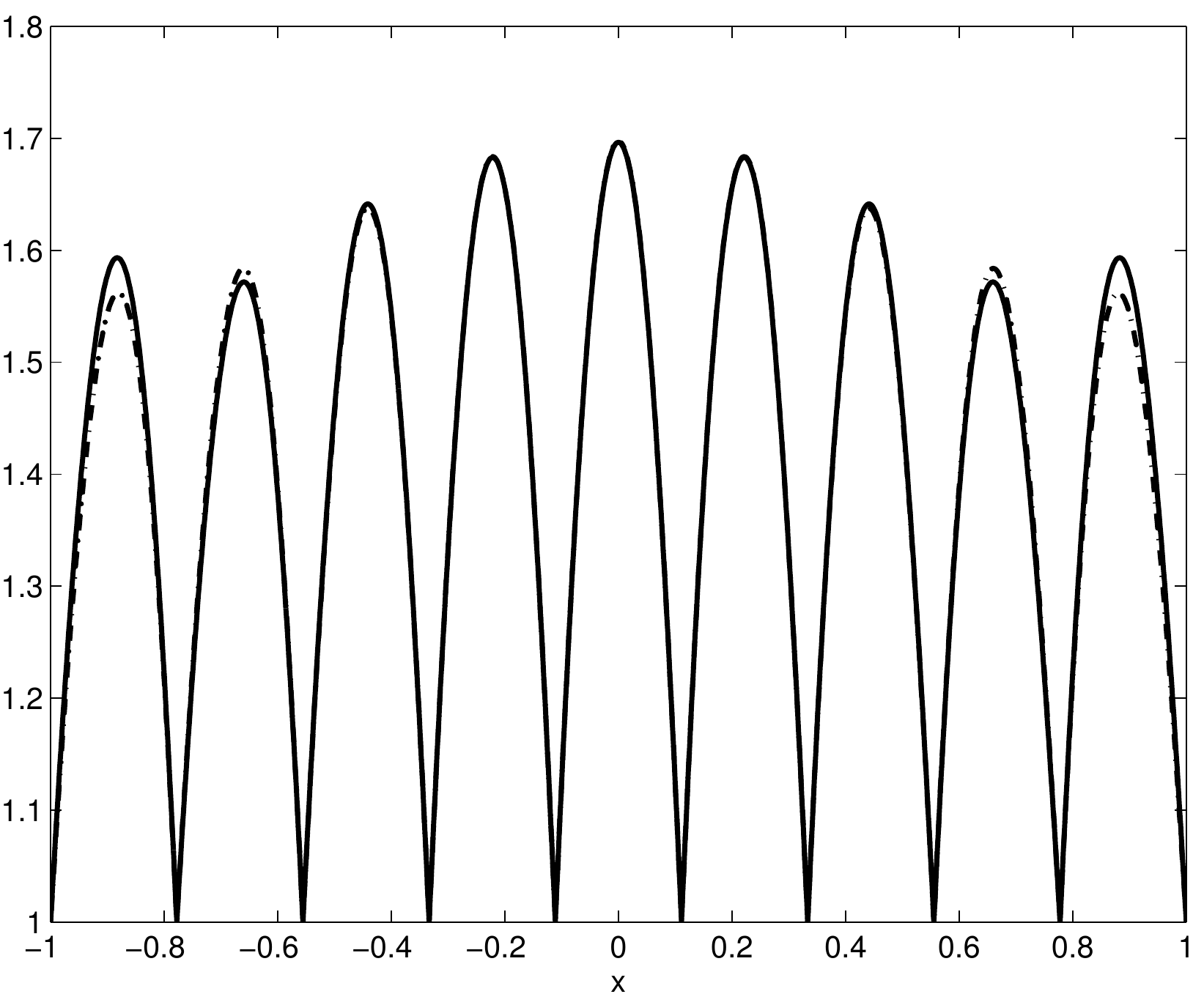}\\
\includegraphics[width=0.45 \textwidth]{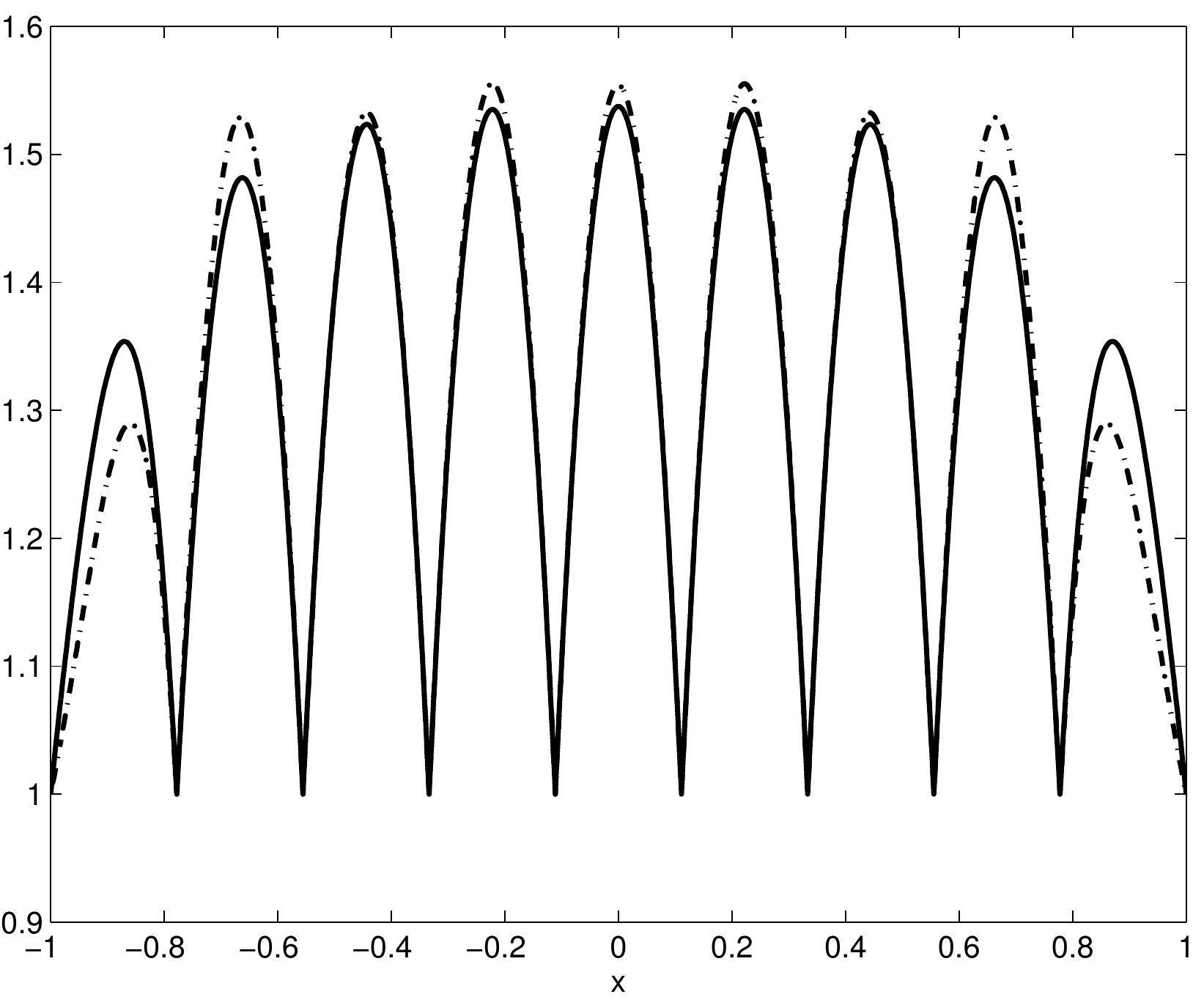}&
\includegraphics[width=0.45 \textwidth]{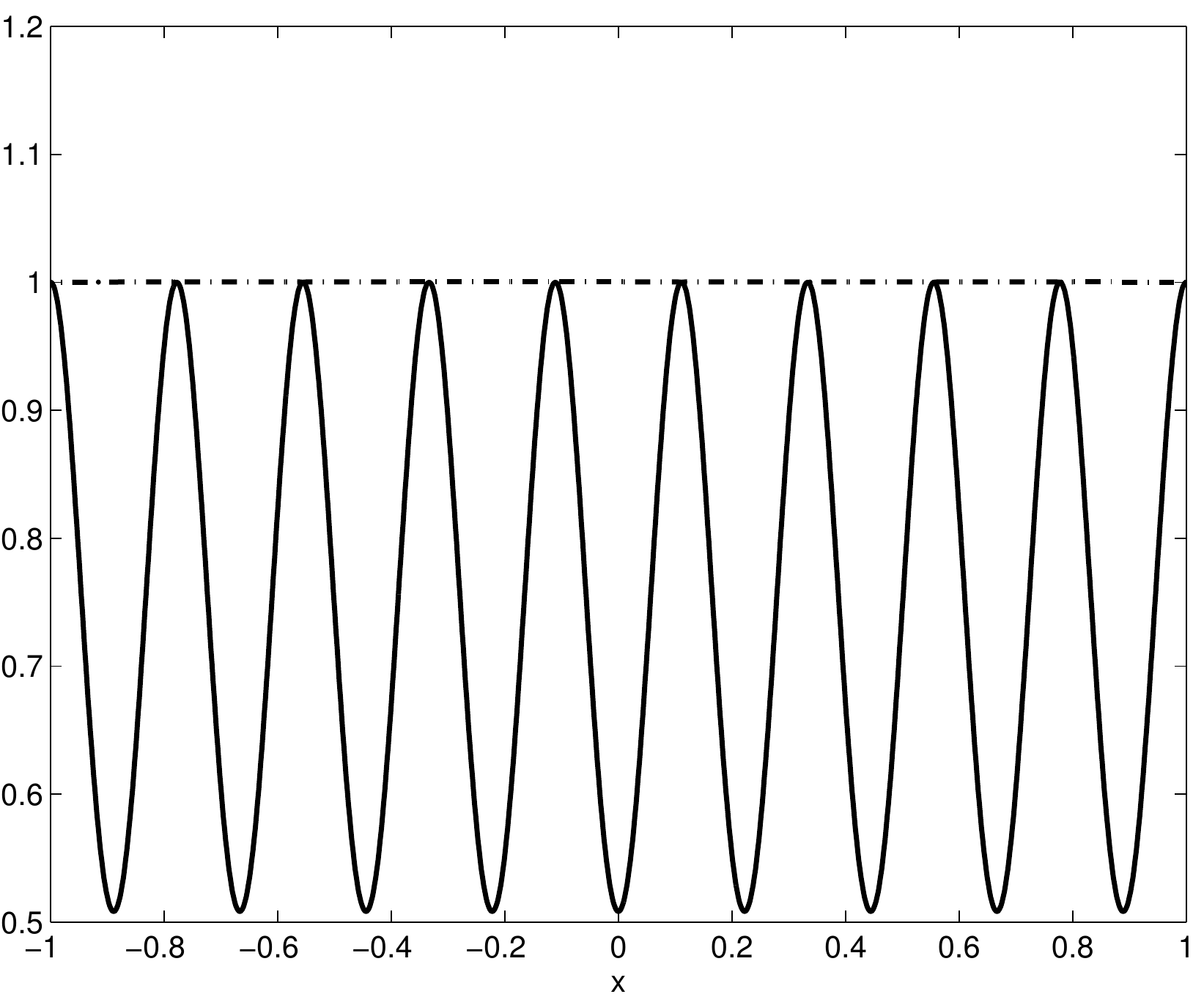}
\end{tabular}
\caption{Comparison between the standard Lebesgue function (solid line) and the 
rescaled Lebesgue function (dotted line) for the $\mathcal C^2$ Wendland kernel on $10$ equally 
spaced points of $[-1,1]$ with different $\varepsilon$. 
From top left to bottom right, $\varepsilon = 0.5, \,1,\, 2,\, 4$.}\label{fig:WenLeb}
\end{figure}

Similar behaviour can be obseverd in the two dimensional setting. In Figure \ref{fig:WenLeb2d_card} 
we show the
comparison between the Lebesgue functions for the Wendand function with standard cardinal functions 
and the rescaled ones on the
cardiod contained in $[-1,1]^2$. 

\begin{figure}[t]
\begin{tabular}{cc}
\includegraphics[scale=0.32]{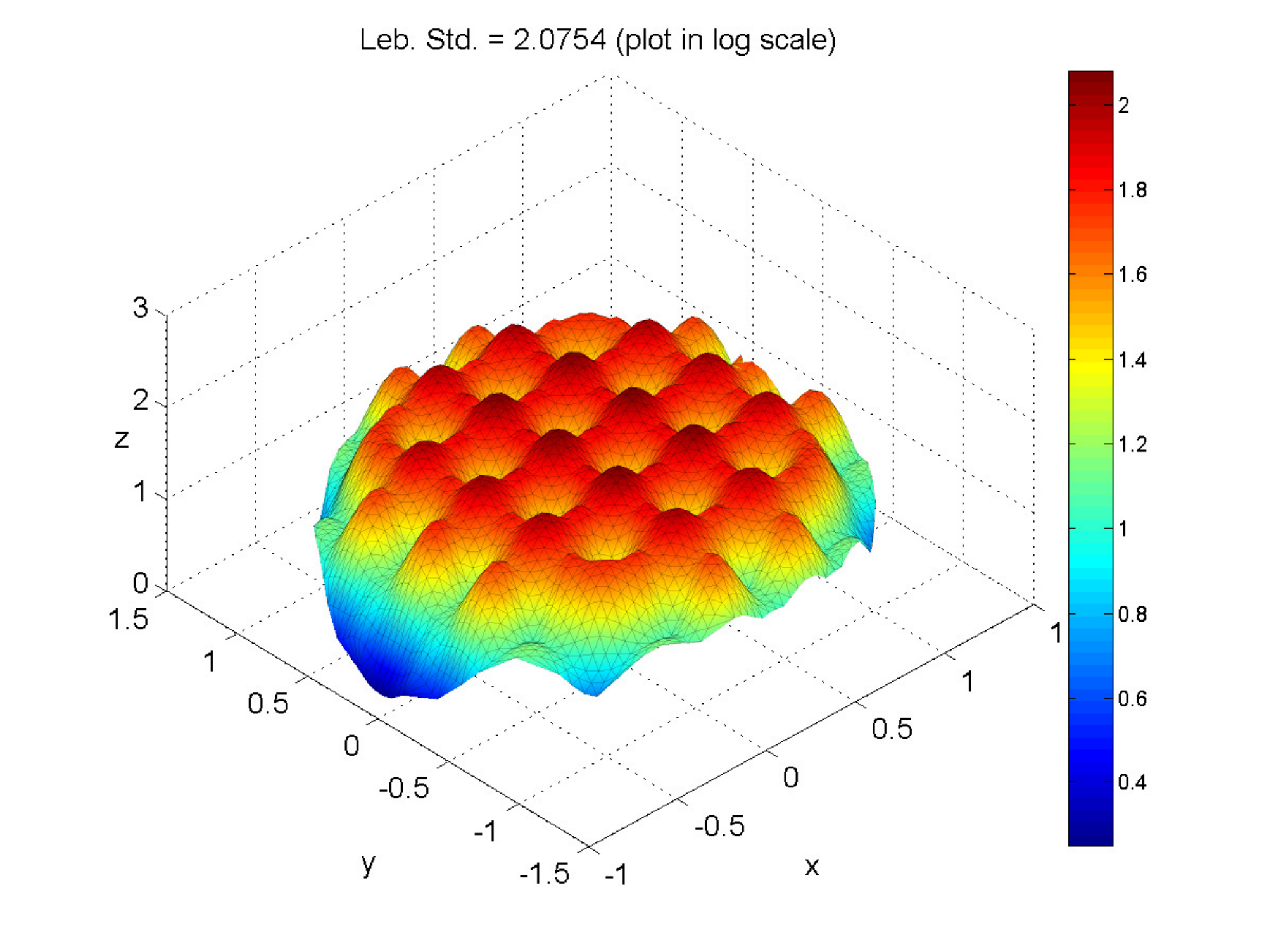} &
\includegraphics[scale=0.32]{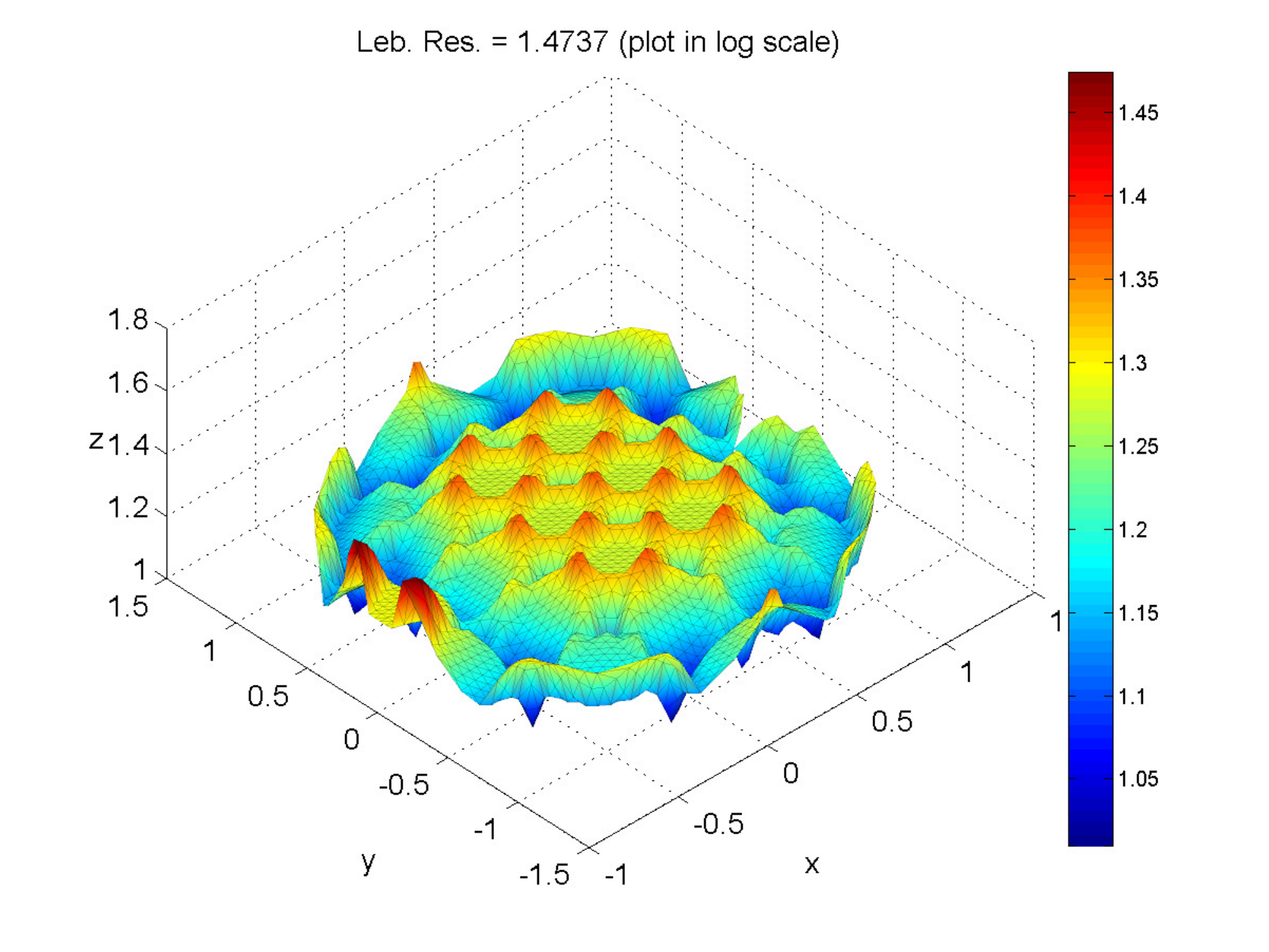}
\end{tabular}
\caption{Comparison between the Lebesgue function with standard basis (Left) and the 
rescaled one (right) for the $\mathcal C^2$ Wendland kernel on the cardiod with $\varepsilon = 
3$.}\label{fig:WenLeb2d_card}
\end{figure}

\begin{figure}[t]
\begin{tabular}{cc}
\includegraphics[scale=0.32]{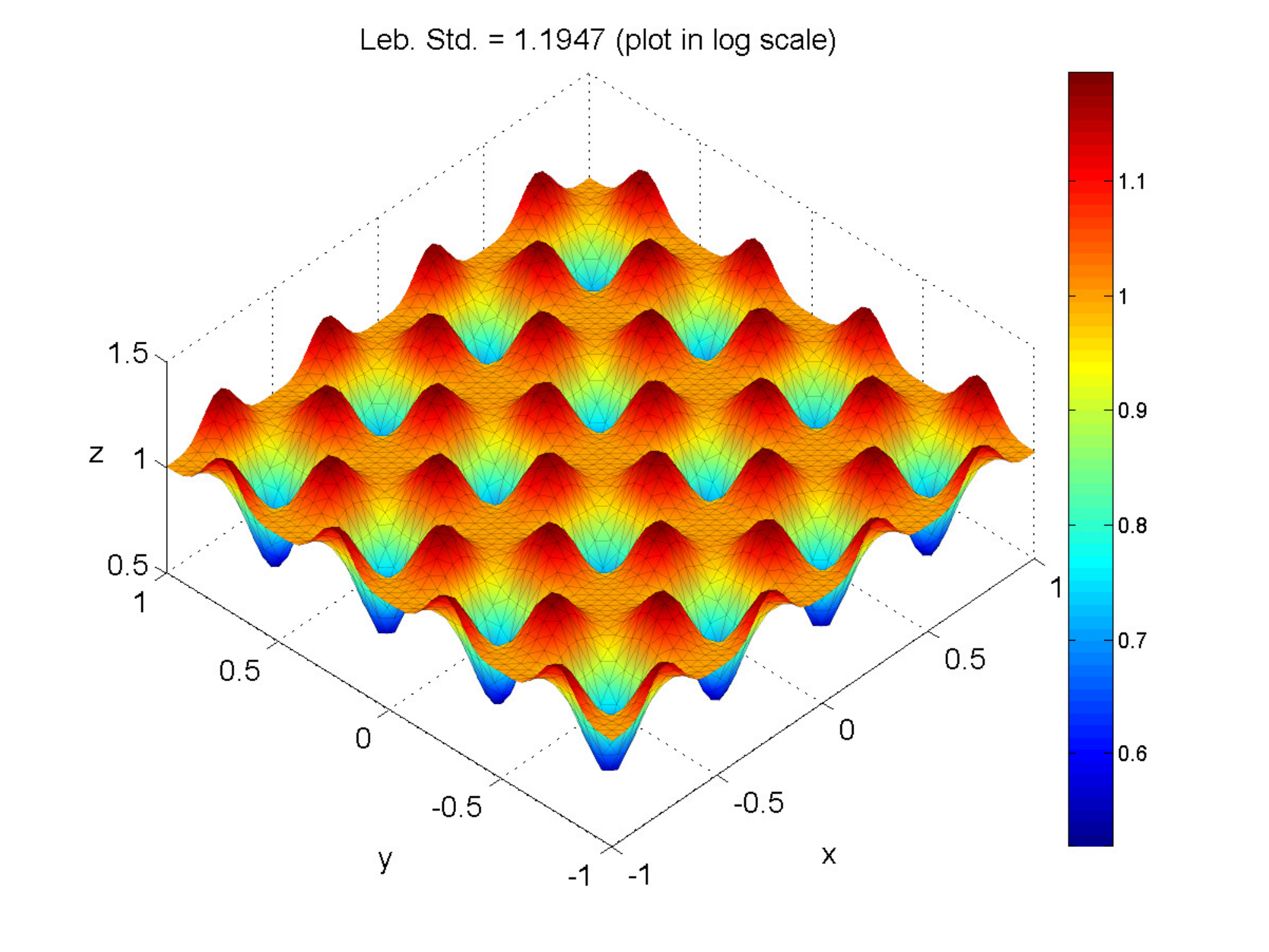} &
\includegraphics[scale=0.32]{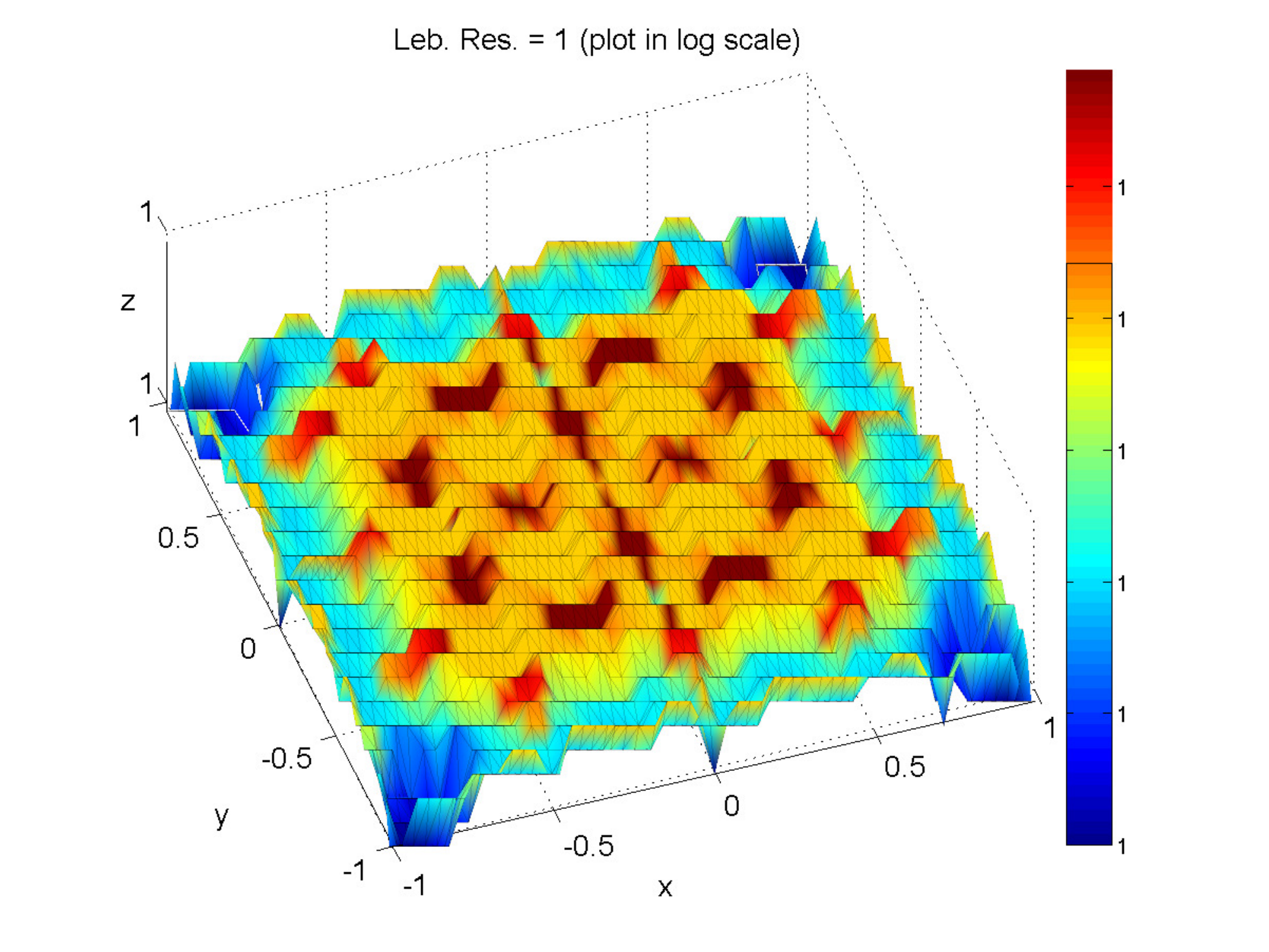}
\end{tabular}
\caption{Comparison between the Lebesgue function with standard basis (Left) and the 
rescaled one (right) for the $\mathcal C^2$ Wendland kernel on the square with $\varepsilon = 
3.85$.}\label{fig:WenLeb2d_square}
\end{figure}

In Figure \ref{fig:WenLeb2d_square} we did a similar test on the square. 
We have chosen $\varepsilon=3.85$ because we wanted to show that bigger values of $\varepsilon$ are 
meaningless. 
The reason of this relies in the support of the Wendland function which is $1/\varepsilon$. In the 
example, we have taken $25$ 
equally spaced points on the square, so that for values of $\varepsilon$ bigger than 
$\varepsilon_M=2$ the cardinal functions have disjoint support. 
Therefore for values of $\varepsilon \ge 2 \varepsilon_M$ we can not
assume that the interpolant of the function $1$ is always not vanishing since some points of the 
domain fall outside the support 
of the cardinal functions (giving a value $0$ of the interpolant). This explains also why in the one 
dimensional case the cardinal functions
with $\varepsilon=4$ give a Lebesgue function identically equal to $1$.


\subsection{Rescaled PUM, accurate PUM and Variably Scaled interpolation}
Here we provide a comparison of the rescaled interpolation with the partition of unity approach  
with the {\it variably scaled} kernel interpolation discussed in \cite{BLRS15}. 

Concerning the application to the PUM, we used the {\it block-based} partition of unity algorithm 
presented in \cite{CDeRP16}, 
that leads to a faster evaluation of the interpolant. Moreover we computed the weight functions
\begin{equation}
w_j({\bs x})={K({\bs x},{\bs  x}_j) \over \sum_{k=1}^N K({\bs x},{\bs  x}_k)}
\end{equation}
where as usual, $K(\cdot,{\bs x}_j)$ is the radial kernel considered.

The first experiment considers the 2d {\it Askley's test function} (well-known for testing 
optimization algorithms \cite{R15})
\begin{equation}
f(x,y)=-20 \;e^{-0.2 \sqrt{0.5(x^2+y^2)}}- e^{-0.5( \cos(2\pi x) + \cos(2\pi y))}+20+e
\end{equation}
interpolated on $10^3$ Halton points on the disk centered in $(0.5,0.5)$ and radius $0.5$ with the 
Wendland kernel W2. As evalution points we took a grid $100\time 100$ uniformly distributed points 
in the convex hull. 
We computed the RMS error at $30$ values of the shape parameter $\varepsilon \in [0.01, 2]$. 
The results are shown in the Figure \ref{fig:PUvsRPU_disk}.
\begin{figure}[t]
\centering
\includegraphics[width=5.5cm, height=5cm]{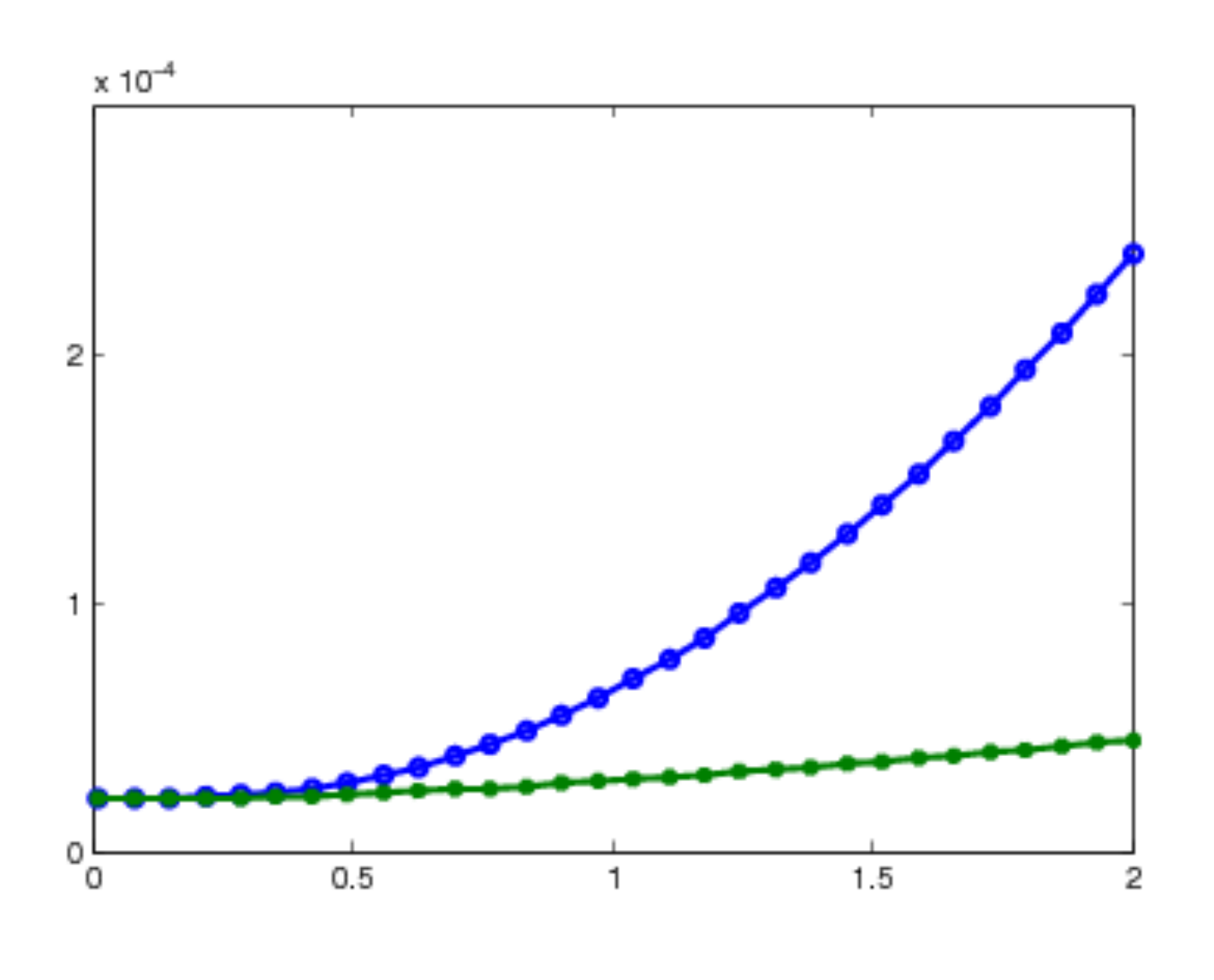}
\includegraphics[width=5.5cm, height=4.9cm]{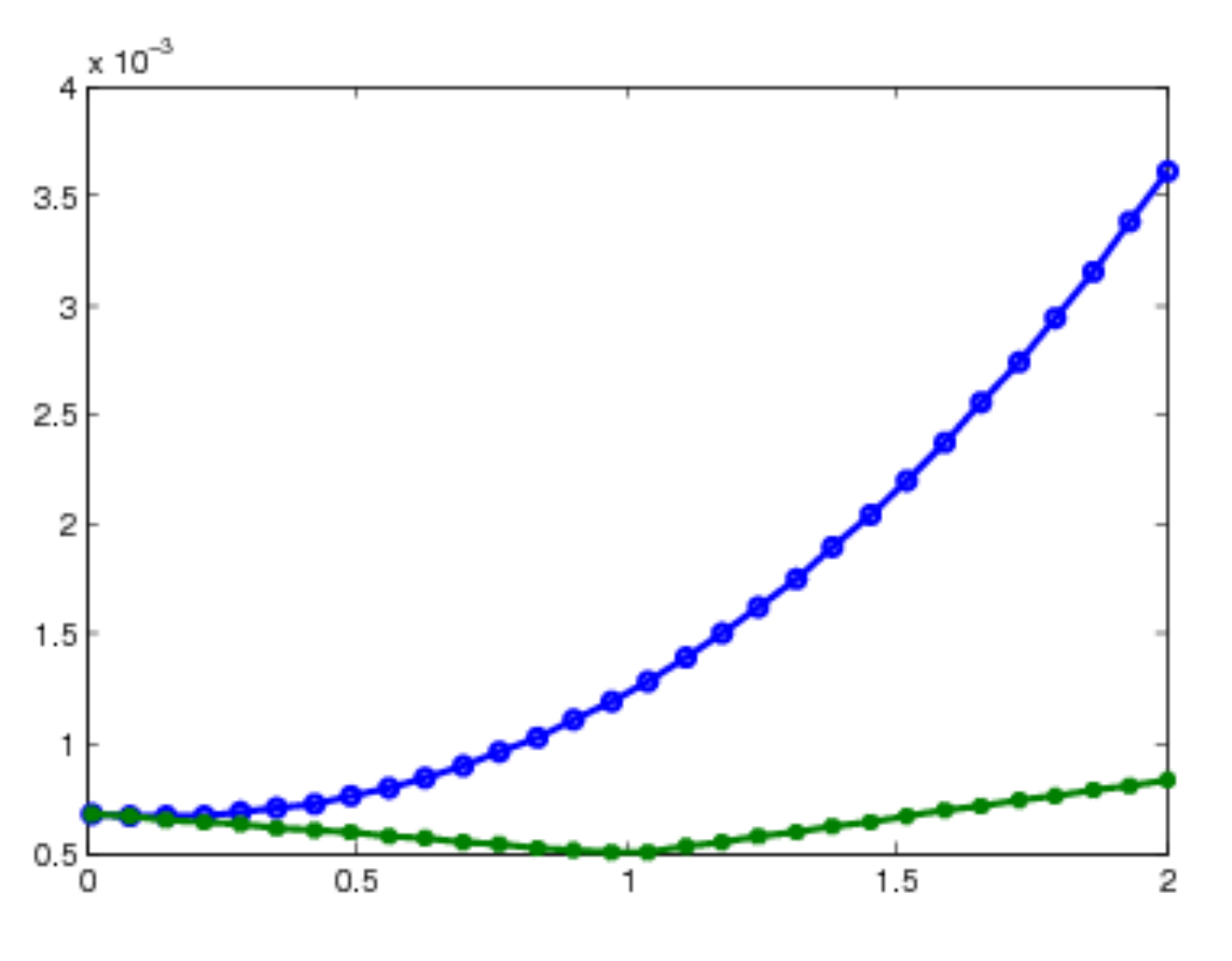}
\caption{RMSE (Left) and MAXERR (Right) for the classical and the rescaled PoU for $\varepsilon 
\in[0.01,2]$ chosen with Trial\&Error. }\label{fig:PUvsRPU_disk}
\end{figure}
The behaviour of the error of the rescaled PU (RPU) interpolant that we notice in Fig 
\ref{fig:PUvsRPU_disk} has been observed in many other examples. Some remarks
\begin{itemize}
 \item the RPU reaches the same precision of the standard PU, but using a ``thinner'' points set;
 \item letting $T_{RPU}$ the evaluation time of the rescaled interpolant and $T_{PU}$ that of the 
classical one, 
 we observed thet $T_{RPU} < c\,T_{PU}$ with $c \approx 1.05$.
\end{itemize}

\vskip 0.2in

In the recent work \cite{DeRPV16} an {\it accurate PUM} (A-PUM), combined with an optimal local RBF 
approximation 
via a {\it priori} error estimates (as done in \cite{CDeRP16}) has been presented. 
The a method enables to select both suitable sizes of the different
PU subdomains and shape parameters, i.e. the {\it optimal couple} $ (r_j^*, \epsilon_j^*)$ for the 
subdomain $\Omega_j$. 
The method uses a bivariate LOOCV strategy, which generalizes to the 2-dimension the classical LOOCV 
(cf. e.g. \cite[\S 17.1.3]{F07}).
and it turns out to be suitable for data with non-homogeneous density. 

\begin{table}[t]
\begin{center}
\begin{tabular}{c|c|c||c|c|c|}
\hline 
 \#DataP& \#EvalP & Method & RMSE & CPU time (sec)\\
 \hline 
 289&1600&A-PUM& 5.72e-4& 3.44 \\
 & & PU& 4.34e-2 & 0.32 \\
 & & RPUM & 1.50e-2 & 0.32\\
 \hline 
 1024&2500&A-PUM& 1.32e-4& 13.66 \\
 & & PU& 1.54e-2 & 0.63 \\
 & & RPUM & 7.55e-3 & 0.66\\
  \hline 
 2500&6400&A-PUM& 6.67e-5& 32.42 \\
 & & PU& 6.14e-3 & 1.32 \\
 & & RPUM & 2.89e-3 & 1.30\\
 \hline
 \end{tabular}
 \caption{Comparison between A-PUM, PUM and RPUM with $\epsilon=5$ with W2, on various grids on the 
square $[0,1]^2$, for
 interpolation of $f(x_1,x_2)=(x_1^2+x_2^2-1)^9$.} \label{Table444}
\end{center} 
\end{table}

Variably Scaled Kernels (VSK) were introduced in \cite{BLRS15} 
with the aim to get more flexibility and better approximation properties with
respect to the fixed scaling, commonly used in RBF interpolation. 
The idea consists in defining a {\it scale function}, say $c: \R^d \longrightarrow [0,\infty)$, 
that 
transforms the interpolation problem with data locations ${\bs x}_j \in \R^d$ to data 
location $({\bs x}_j, c({\bs x}_j)) \in \R^{d+1}$ and then use a fixed-scale kernel on $\R^{d+1}$. 
Following \cite{BLRS15}, a VSK can be defined 
in a general way by introducing a scale function (instead of a scalar shape parameter). 
\begin{definition}
Let $c : \R^d \rightarrow \R$ be a scale function, then a variably scaled kernel associated to a 
kernel $K$ on $\R^{d+1}$ is
$K_c({\bs x}, {\bs y}):=K(({\bs x}, c({\bs x}), ({\bs y}, c({\bs y}))$ .
\end{definition}
By construction $K_c$ is on $\R^d$. We recall also two results, 
whose proofs are in \cite{BLRS15}, showing that the new kernels preserve the same properties as the 
original one
\begin{theorem}
If $K$ is strictly positive definite on $\R^{d+1}$ so is $K_c$ on $\R^d$.
\end{theorem}
In fact the matrix $K_c({\bs x}_i, {\bs x}_j)=K( ({\bs x}_i, c({\bs x}_i), ({\bs x}_j, c({\bs 
x}_j))$
is still strictly positive definite whenever $K$ is.

Notice that in the case of {\it radial} kernel, $\Phi$, the new kernel takes the form
$K_c({\bs x}, {\bs y}):=\Phi(\| {\bs x} - {\bs y} \|^2 + (c({\bs x}) - c({\bs y}) )^2 )$
that reduces to the classical one if we take a constant scale function.

Consider the map 
$ \sigma : {\bs x} \to ({\bs x}, c({\bs x}))$
from $\R^d$ into a $d$-dimensional submanifold $\sigma(\R^d)$ of $\R^{d+1}$. 
If we then consider a discrete set $X=\{x_1,\ldots, x_N\} \subset \Omega \subset \R^d$, 
then $\sigma(X) \subset \sigma(\Omega) \subset \sigma(\R^d) \subset \R^{d+1}$. Hence, the 
interpolant becomes
\begin{equation}
P_{\sigma,f,X}({\bs x})=P_{1,f,\sigma(X)}({\bs x}, c({\bs x}))= P_{1,f,\sigma(X)}(\sigma({\bs 
x}))\,.
\end{equation}
This says that we can consider the interpolant $P_{1,f,\sigma(X)}$ at scale $1$ of the data of $f$ 
at the points $({\bs x}_j, c({\bs x}_j)), \; j=1,\ldots,N$ of $\sigma(X)$. This also means that in 
$\R^{d+1}$ we use the kernel $K_c$ and if
we project the points $({\bs x}, c({\bs x}) \in \R^{d+1}$ back to ${\bs x} \in \R^d$, the projection 
of the kernel $K_c$ on $\R^{d+1}$ turns
into a variable shape kernel on $\R^d$ if the shape function $c$ is not constant.

In particular the error analysis and stability of this variably scaled problem in $\R^d$ is 
the same of the one with a fixed scale on a submanifold $\sigma(\R^d)$ of $\R^{d+1}$.

The VSK approach has another interesting property
\begin{theorem}
The native spaces ${\cal N}_K$ and ${\cal N}_{K_c}$ are isometric. {\rm The proof is at \cite[p. 
204]{BLRS15}}.
\end{theorem}
\vskip 0.2in

\noindent For $d=2$ we performed the following experiment. Consider the Franke test function 
(\ref{f12d}), 
the compactly supported W2 sampled on $200$ equally spaced points of half unit sphere centered in 
$(0,0,0)$.
The nodes in $\R^2$ are the projections on the unit disk of the previous ones, so that $c({\bs 
x})=\sqrt{1-x_1^2-x_2^2}$.
The evaluation points in $\R^2$ are obtained by restricting the grid $100 \times 100$ of the square 
$[-1,1]^2$ to the unit disk, 
while the points in $\R^3$ are obtained by the map 
$\sigma({\bs x})=({\bs x}, c({\bs x}))$ (in Figure \ref{fig:points_halfsphere} we show only 100 
points). The shape parameter is $\varepsilon=5$. 
In Table \ref{Table3} we report the results of the corresponding RMSE while in Table \ref{Table4} 
the max err.
Similar results, for this example, can be observed with different values of the shape parameter. 
In Figure \ref{fig:rmse_comparison_all} we plot the RMSE in the stationary case for the three 
methods taking $20$ values of $\varepsilon \in [0.1,5]$.

\begin{figure}[t]
\begin{center}
\includegraphics[scale=0.5]{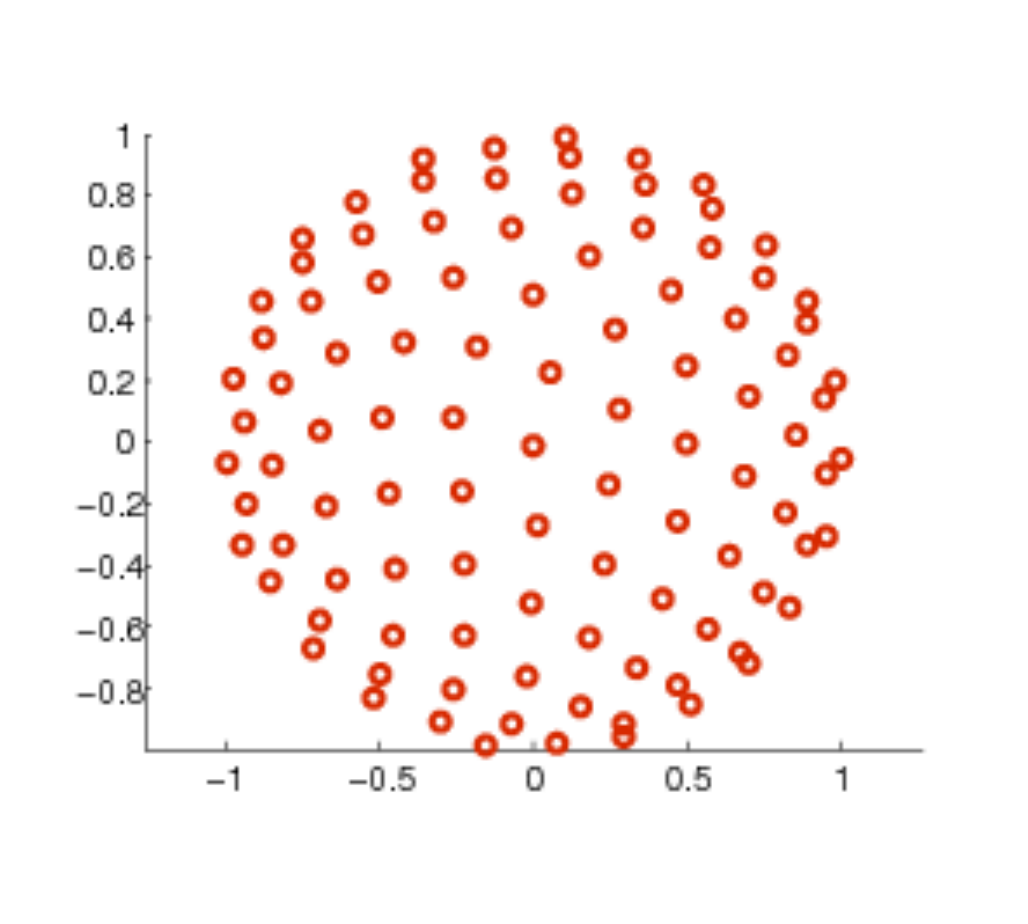} \\
\includegraphics[scale=0.5]{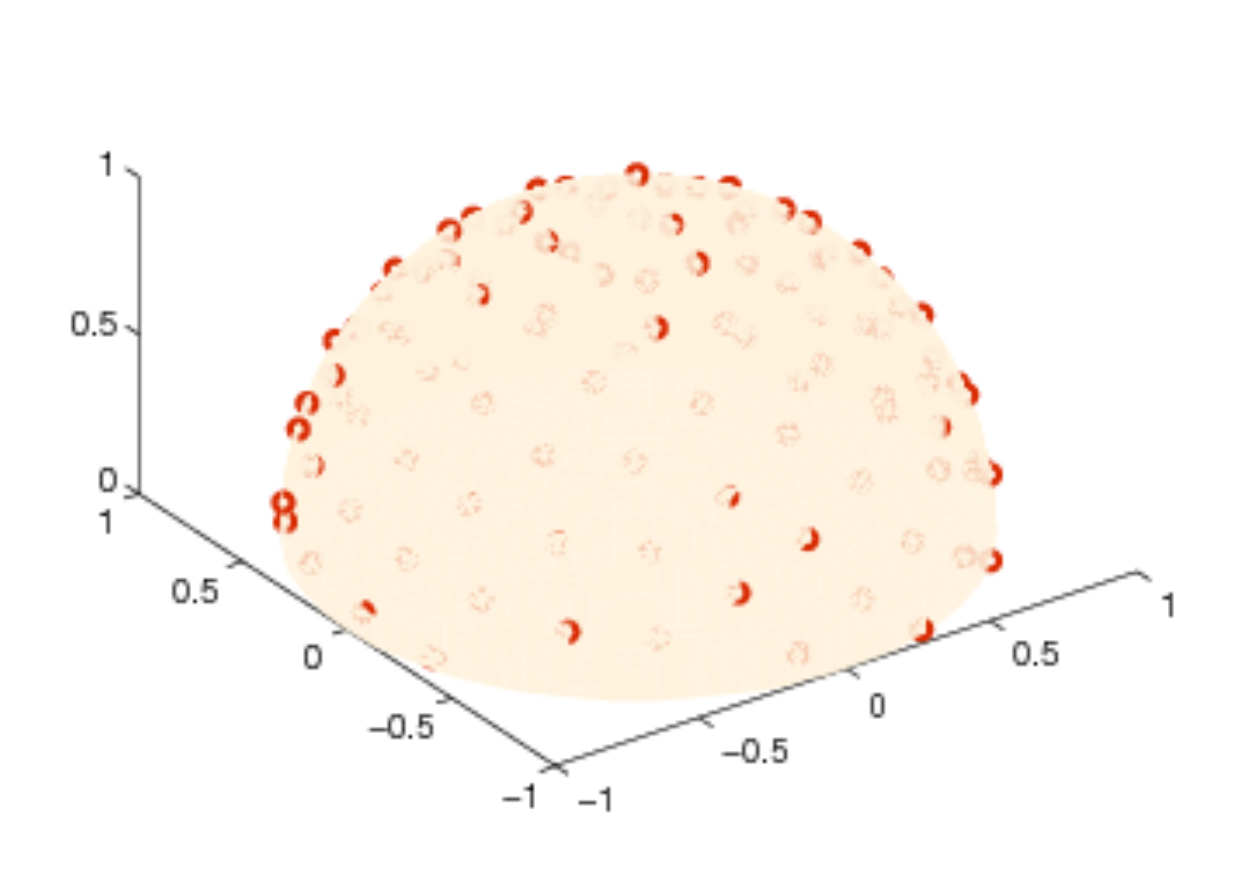}
\caption{The points of the VSK example }\label{fig:points_halfsphere}
\end{center}
\end{figure}
\begin{figure}[t]
\begin{center}
\includegraphics[width=9cm, height=7cm]{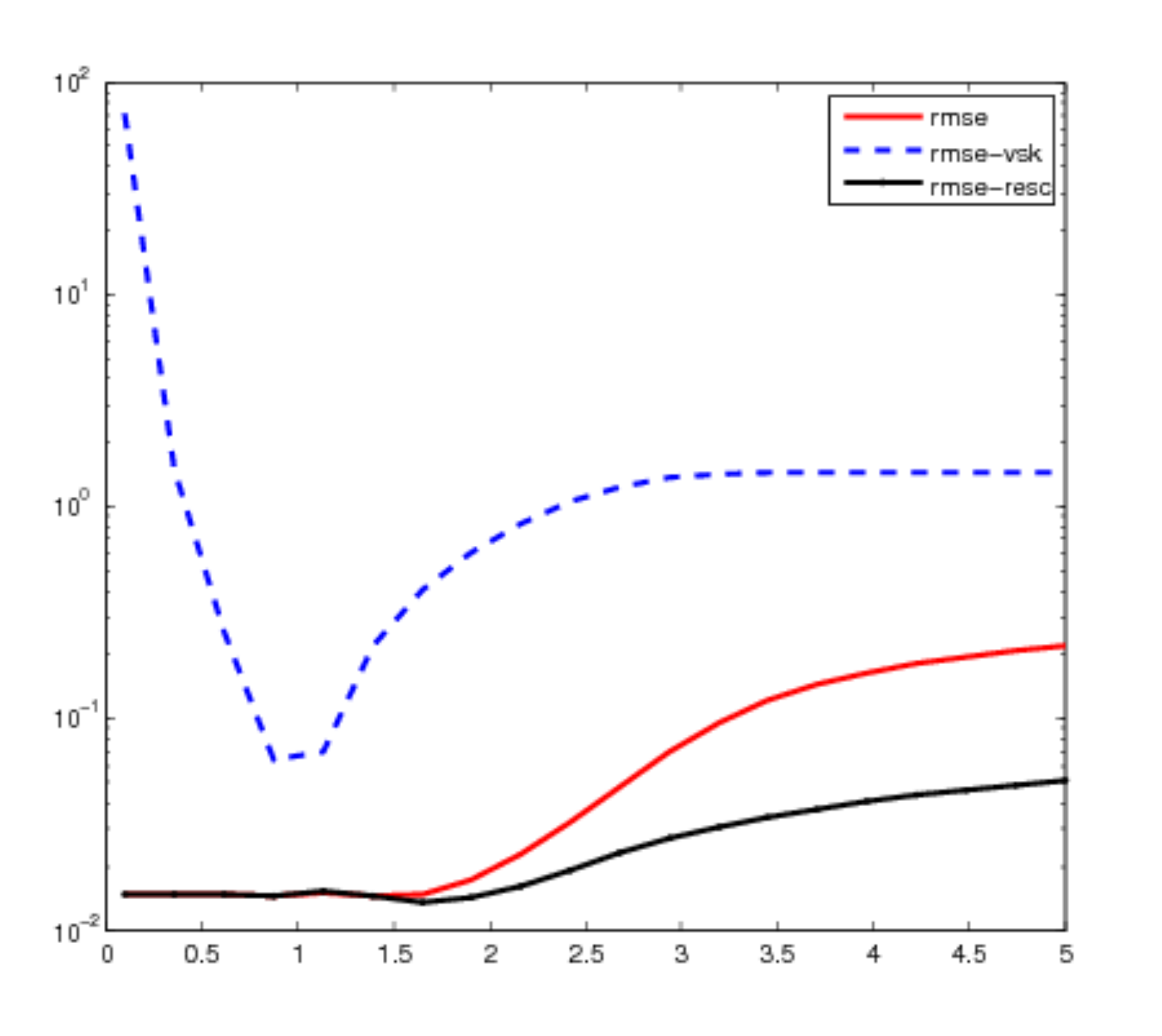} 
\caption{Classical, VSK and Rescaled method for W2 on varying the shape parameter 
}\label{fig:rmse_comparison_all}
\end{center}
\end{figure}
\begin{table}
\begin{center}
\caption{RMSE with and without rescaling applied to variably scaled kernel} \label{Table3}
\begin{tabular}{p{2cm}p{4.0cm}p{2.5cm}}
\hline\noalign{\smallskip}
 & Standard & +R \\
\noalign{\smallskip}\hline\noalign{\smallskip}
 Standard &  $2.2e-01$ &$5.1e-02$ \\ 
+VS& $1.5e+00$& $9.6e-02$ \\
\noalign{\smallskip}\hline\noalign{\smallskip}
 \end{tabular}
 \end{center}
\end{table} 
\begin{table}
\begin{center}
\caption{Max err with and without rescaling applied to variably scaled kernel} \label{Table4}
\begin{tabular}{p{2cm}p{4.0cm}p{2.5cm}}
\hline\noalign{\smallskip}
 & Standard & +R \\
\noalign{\smallskip}\hline\noalign{\smallskip}
 Standard & $9.7e-01$ & $4.2e-01$\\ 
 +VS & $3.2e+00$ & $5.3e-01$ \\
\noalign{\smallskip}\hline\noalign{\smallskip}
 \end{tabular}
 \end{center}
\end{table} 

\section{Conclusions}
In this paper we have studied more deeply the rescaled localized radial basis function interpolant 
introduced in \cite{DFQ14}. We have proved
that this interpolant gives a partition of unity method that reproduces constant functions, exactely 
as does the Shepard's method. 
One feature is that the shape parameter of the kernel can be chosen in a safe range, not too small 
and relatively large, avoiding some
numerical instability tipical that occours when the shape parameter is too small, or a severe 
ill-conditioning in the opposite case.  
The method performs better, in terms of error and CPU time, than a classical PUM. Moreover if it is 
coupled with a variably scaled kernel
strategy it allows to control the errors and the condition number of the system (as shown in Table 
\ref{Table4}).

\vskip 0.2in
\textbf{Acknowledgement:} This work has been supported by the INdAM-GNCS funds 2016 and by the ex 
60\% funds, year 2015, of the University of Padova. 


\end{document}